\documentclass[a4paper,12pt]{article}

\usepackage{pstricks}
\usepackage{graphicx,psfrag}
\usepackage{amsmath}
\usepackage{amssymb}
\usepackage{amsthm}

              \newcommand {\ve}   {\varepsilon}

      \newcommand {\RRR}  {{\mathbb R}}

          \newcommand {\AAA}  {\mathcal{A}}
            
     \newcommand {\beq}  {\begin{equation}}
      \newcommand {\eeq}  {\end{equation}}

      \newtheorem{theorem}{Theorem}
      
      \newtheorem{proposition}{Proposition}
      \newtheorem{remark}{Remark}
      \newtheorem{opr}{Definition}

\author{Alexander Plakhov\thanks{Department of Mathematics,
University of Aveiro, Aveiro 3810-193, Portugal} \and Vera Roshchina\thanks{CIMA, University of \'{E}vora,  Portugal; Ci\^{e}ncia 2008}}

\title{Fractal bodies invisible in 2 and 3 directions}

\date{}

\begin{document}

\maketitle

\begin{abstract}
We study the problem of invisibility for bodies with a mirror
surface in the framework of geometrical optics. We show that
for any two given directions it is possible to construct a
two-dimensional fractal body invisible in these directions.
Moreover, there exists a three-dimensional fractal body
invisible in three orthogonal directions. The work continues
the previous study in \cite{0-resist,PlakhovRoshchina2011},
where two-dimensional bodies invisible in one direction and
three-dimensional bodies invisible in one and two orthogonal
directions were constructed.
\end{abstract}

\begin{quote}
{\small {\bf Mathematics subject classifications:} 37D50, 49Q10
}
\end{quote}

\begin{quote}
{\small {\bf Key words and phrases:} Billiards, invisible bodies, shape optimization, geometrical optics, problems of minimal resistance.}
\end{quote}

\section{Introduction}

Invisibility has fascinated people's imagination since ancient
times: the idea is exploited in folklore, fiction and movies.
This ``magical'' concept is, however, rapidly migrating into
the scientific domain. On the cutting edge of the modern
developments is the design of metamaterials with special
refractive properties, which could, amongst other important
applications, ultimately lead to the creation of a real
invisible cloak. For an overview of the recent works in this
field we refer the reader to our recent article
\cite{PlakhovRoshchina2011}. The aforementioned developments
deal with the wave nature of light, and metamaterials are
engineered at the nanoscale level. The effects specific to
geometrical optics, however, also remain important in modern
technology, mostly in cases where the objects are large enough
for geometrical optics to dominate the wave effects. Examples
include fiber optics, design of lenses (e.g. for photography or
DVD readers), and many others.

In this article we are concerned with invisibility in
billiards. We consider bodies with a perfectly mirrored surface
in a beam of light, or, equivalently, in a flow of
non-interacting billiard particles. Invisibility in a direction
$v$ (where $v$ is a unit vector) means that any light ray which
initially moves along a straight line in this direction, after
several reflections from the body's surface will eventually
move along the same straight line. Invisibility in a set of
directions means that the above is true for any direction from
this set. This problem is closely related to the problem of
minimal resistance going back to Newton \cite{N}. The latter
consists of finding a body, from a given class of bodies, that
experiences the smallest possible force of flow pressure, or
resistance force. Since the 1990s, many interesting results on
this problem have been obtained by various authors (see, e.g.,
\cite{BB,BK,CL1,LO,LP1,RMS_review,ESAIM}). Bodies of zero
resistance in one and two directions are described in
\cite{0-resist} and \cite{PlakhovRoshchina2011} respectively.
In both cases it is possible to construct an invisible body by
arranging several such bodies together in a specific way.  In
\cite{PlakhovRoshchina2011} it was shown that there exist
bodies invisible in two mutually orthogonal directions in the
three-dimensional setting and that bodies invisible in all
directions do not exist.

In this work we continue the study of invisibility and
construct bodies invisible in any two directions in
two-dimensional space and in three orthogonal directions in
three dimensions. Each body in the construction is a union of
infinitely  many pieces of varying size going to zero, where
each piece is a domain with a piecewise smooth boundary. By
slightly abusing the terminology, such a union will be called a
{\it fractal body}, or a {\it solid fractal body}. In a
preliminary construction (Section \ref{ss:2dir prelim}) and in
a limiting case of the basic construction (Section
\ref{s:3dir}) some pieces comprising the body are smooth curves
(in the 2D case) or pieces of surfaces (in the 3D case). The
corresponding body will be called a {\it thin fractal body}.

The article is organized as follows: we first reintroduce some
definitions and briefly revisit earlier results in
Section~\ref{s:prelim}. In Sections~\ref{s:2dir} and
\ref{s:3dir} we explain the construction of bodies invisible in
two and three directions respectively. Section~\ref{s:summary}
contains some final remarks and a brief discussion of open
problems.

\section{Bodies of zero resistance and invisible bodies}\label{s:prelim}

Consider a parallel flow of point particles in $\RRR^n$ moving
with unit velocity $v \in S^{n-1}$ towards a body $B$ at rest.
The flow is so rarefied that the particles do not mutually
interact.  Particles reflect elastically when colliding with
the body's surface and move freely between consecutive
collisions.

We deal with bounded (not necessarily connected) bodies composed of a (possibly infinite) number of piecewise smooth fragments.
   When a particle moving along a straight line and with a constant
velocity hits the boundary of $B$, it is reflected from the
latter without loss of speed, and keeps moving along the new
linear trajectory until the next collision. All reflections are
specular: the angle of incidence just before the collision is
equal to the angle of reflection just after the collision (see
Fig.~\ref{fig:scatt}).
\begin{figure}[h]
\begin{picture}(0,175)
\rput(7.7,3){
\psecurve[linecolor=brown,fillstyle=solid,fillcolor=yellow]
(1.2,-1.5)(1,-0.5)(2.5,1)(1.5,2)(0.3,1.75)(-0.5,2)(-1.6,1)(-0.7,-0.4)
(-1.1,-2.2)(1.2,-1.5)(1,-0.5)(2.5,1)
\psline[linecolor=red,arrows=->,arrowscale=2](-2.985,0.3)(-2,0.3)
\rput(-3.2,0.3){$\xi$}
\rput(-1.8,-2.75){$\xi^+$}
\rput(-1.6,-2){$v^+$}
\rput(-2.5,0.55){$v$}
\rput(0.1,1){\Huge $B$}
\psline[linecolor=red,arrows=->,arrowscale=2](-2,0.3)(-1.2,0.3)(-1.2,-1.8)(-1.65,-2.48)
}
\end{picture}
\caption{ The broken line through the points $\xi$ and $\xi^+$ is a billiard trajectory in the complement of $B$.
    }
\label{fig:scatt}
\end{figure}
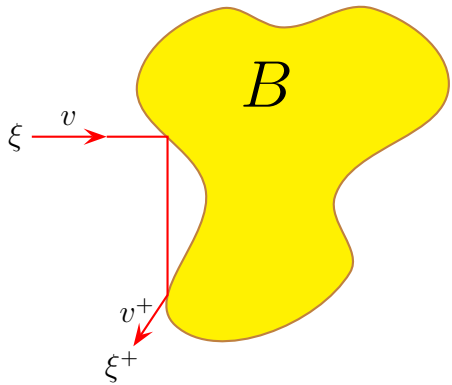
In general, it is possible that the particle never leaves the
body and keeps bouncing off its sides infinitely; however, we
only consider such bodies and velocities of incidence $v$ for
which almost every particle makes a finite number of
reflections. Also note that in some cases the particle may hit
a singular point of the boundary. In this case the further
movement of the particle is not defined. We consider such
bodies and velocities $v$ that the set of initial points $\xi$
for which the motion is undefined (i.e. such that a particle
starting the movement at a point $\xi$ in the direction $v$,
eventually hits a singular point) has zero measure.

In view of this description, for almost any $\xi \in \RRR^n$, the particle that initially moves freely according to $x(t) = \xi + vt$, after a finite number of reflections from $B$ moves freely again according to $x(t) = \xi^+ + v^+t$, where $\xi^+ = \xi^+_{B,v}(\xi)$ and $v^+ = v^+_{B,v}(\xi)$ are measurable functions defined almost everywhere.

\begin{opr}\label{o 2} Let a body $B\subset \RRR^n$. 

\begin{itemize}
  \item[(i)] We say that $B$ {\em has zero resistance in the direction} $v$, if $v^+_{B,v}(\xi) = v$ for all $\xi$ in the domain of $v^+_{B,v}$       (see Fig. \ref{fig zerores invis}\,(a)).
  \item[(ii)] We say that the body $B$ {\it is invisible in the direction} $v$, if it has zero resistance in this direction and, additionally, $\xi^+_{B,v}(\xi) - \xi$ is always parallel to $v$ (see Fig.~\ref{fig zerores invis}\,(b)).
  \item[(iii)]Let $A \subset S^{n-1}$. The body $B$ is said to be {\it invisible/have zero resistance in the set of directions} $A$, if it is invisible/has zero resistance in any direction $v \in A$.
\end{itemize}
\end{opr}
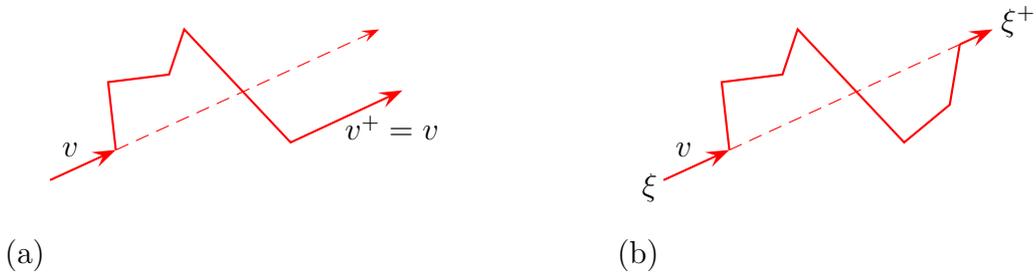
\begin{figure}[h]
\begin{picture}(0,160)
\rput(4,3){
\scalebox{1}{
\psline[linecolor=red,arrows=->,arrowscale=2](-2.165,-1)(-1.299,-0.6)
\psline[linecolor=red,linestyle=dashed,linewidth=0.4pt,arrows=->,arrowscale=2](-1.299,-0.6)(2.165,1)
\psline[linecolor=red,arrows=->,arrowscale=2](-1.299,-0.6)(-1.4,0.3)(-0.6,0.4)(-0.4,1)(1,-0.5)(2.47,0.19)
\rput(2.33,-0.3){$v^+=v$}
\rput(-1.9,-0.6){$v$}
\rput(-2.5,-2){(a)}
\rput(8,0){
\psline[linecolor=red,arrows=->,arrowscale=2](-2.165,-1)(-1.299,-0.6)
\psline[linecolor=red,arrows=->,arrowscale=2](1.732,0.8)(2.165,1)
\psline[linecolor=red,linestyle=dashed,linewidth=0.4pt](-1.299,-0.6)(1.732,0.8)
\psline[linecolor=red](-1.299,-0.6)(-1.4,0.3)(-0.6,0.4)(-0.4,1)(1,-0.5)(1.6,0)(1.732,0.8)
\rput(-1.9,-0.6){$v$}
\rput(-2.35,-1.1){$\xi$}
\rput(2.5,1.1){$\xi^+$}
\rput(-2.5,-2){(b)}
}
}}
\end{picture}
\caption{A typical billiard path in the case of a body (a) having zero resistance in the direction $v$; (b) invisible in the direction $v$. The body is not shown in both cases.}
\label{fig zerores invis}
\end{figure}

Observe that invisibility is a symmetric notion, i.e. if a body is invisible in a direction $v$, it is also invisible in $-v$. This follows directly from the fact that billiard dynamics is time-reversible.

Examples of bodies invisible in one direction constructed with the use of thin mirrors can be found on the Wikipedia page on invisibility \cite{Wiki_invis}, and one of them is reproduced in Fig.~\ref{fig:prelim01}\,(a).
\begin{figure}[h]
\centering
\begin{picture}(420,110)
\put(30,30){\includegraphics[height=80pt, keepaspectratio]{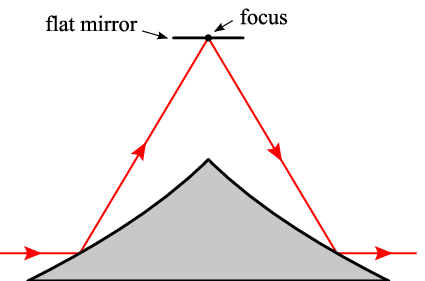}}
\put(190,30){\includegraphics[height=80pt, keepaspectratio]{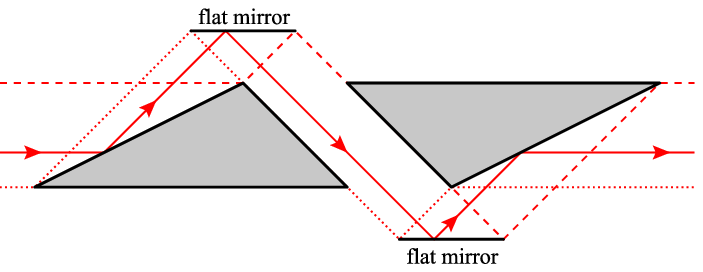}}
\put(30,10){(a)}
\put(190,10){(b)}
\end{picture}
\caption{Bodies invisible in one direction: (a) using two parabolic mirrors and a thin flat mirror; (b) using flat mirrors.}
\label{fig:prelim01}
\end{figure}
The German Wikipedia page \cite{Wiki_invis_de} has got more
interesting examples and videos of prototypes designed by Karl
Bednarik (one of them that uses only flat mirrors is plotted in Fig.~\ref{fig:prelim01}\,(b)).

A solid body (i.e. without use of any thin mirrors) of zero
resistance in one direction was constructed in \cite{0-resist}
(see Fig.~\ref{fig:prelim02}).
\begin{figure}[h]
\centering
\begin{picture}(420,140)
\put(50,30){\includegraphics[height=110pt,
keepaspectratio]{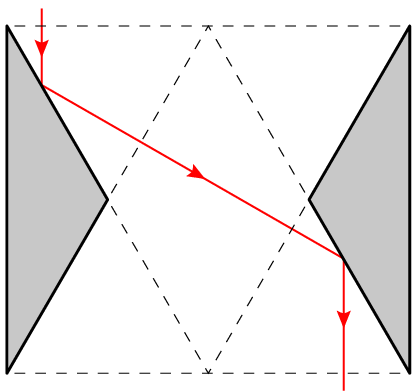}}
\put(240,35){\includegraphics[height=100pt,
keepaspectratio]{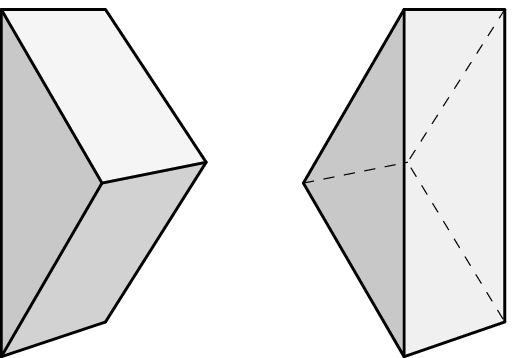}}
\put(50,10){(a)}
\put(240,10){(b)}
\end{picture}
\caption{Solid body of zero resistance in one direction: (a) two-dimensional construction; (b) three-dimensional version.}
\label{fig:prelim02}
\end{figure}
It was also shown in the same work that there exist connected (and even homeomorphic to the ball) bodies of zero resistance in one direction. An invisible body is obtained by using two such bodies consecutively (see \cite{0-resist} for details).

In \cite{PlakhovRoshchina2011} a body of zero resistance in two
directions in a three-dimensional setting was described. The
body is sketched in Fig.~\ref{fig:prelim03}: it employs 8
fragments of congruent parabolic cylinders with two mutually
orthogonal focal lines (each line corresponds to 4 fragments).
An invisible body is constructed by putting 4 such bodies
together, as shown in Fig.~\ref{fig:prelim03}\,(b). For more
details we refer the reader to \cite{PlakhovRoshchina2011}.
\begin{figure}[h]
\centering
\begin{picture}(420,160)
\put(40,30){\includegraphics[height=130pt,
keepaspectratio]{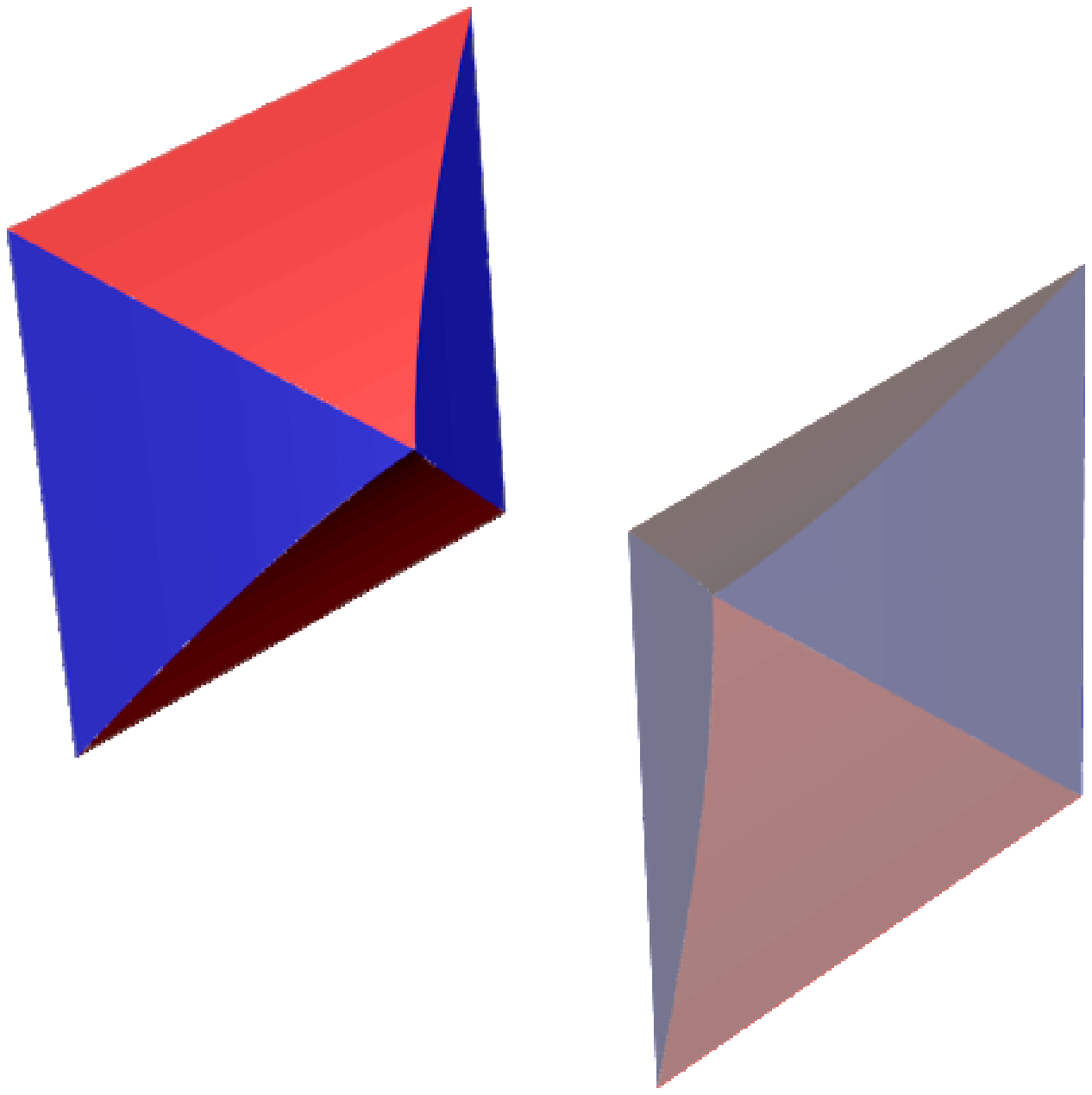}}
\put(250,30){\includegraphics[height=130pt,
keepaspectratio]{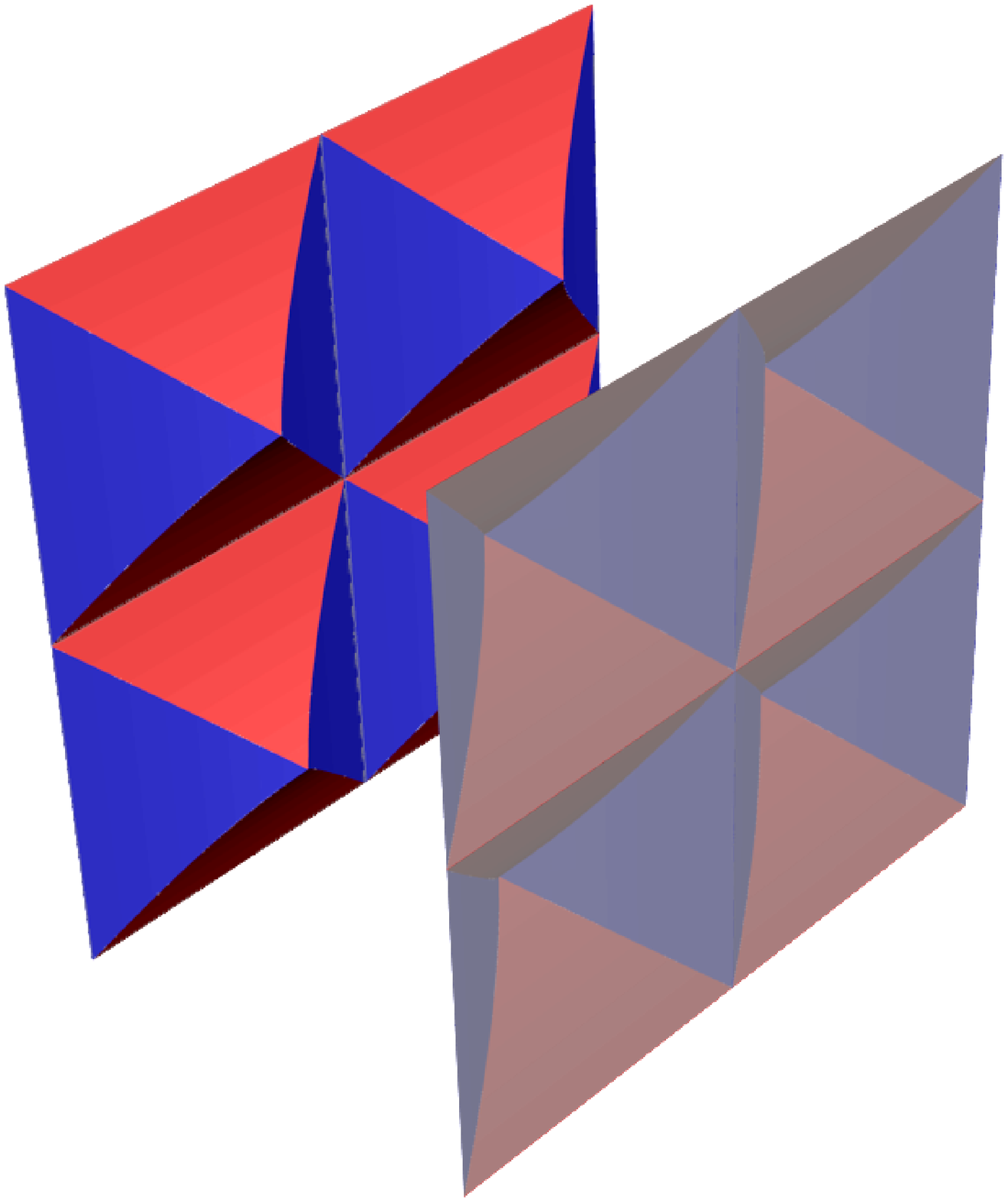}} \put(50,10){(a)}
\put(250,10){(b)}
\end{picture}
\caption{Solid body of zero resistance in two directions.}
\label{fig:prelim03}
\end{figure}

\section{Body invisible in two directions}\label{s:2dir}

\subsection{A thin fractal body invisible in two orthogonal directions}\label{ss:2dir prelim}

We start with a two-dimensional body invisible in two
orthogonal directions. For the clarity of exposition, we assume
that the directions of invisibility are parallel to the $x$-
and $y$-axes.

We construct our body inside the square
$$
S = [-1,1]\times [-1,1] = \{(x,y)\,:\, |x|\leq 1,\ |y|\leq 1\}.
$$
Consider a thin parabolic mirror (i.e. having zero thickness)
$p_1$ given by the equation
$$
p_1 = \bigl\{(x,y)\,:\, y = \frac{1}{2}x^2+\frac{1}{2},\ |x|\leq 1\bigr\}.
$$
Observe that $(-1,1)\in p_1$, $(1,1)\in p_1$, and the focus of
the parabola $y = \frac{1}{2}x^2+\frac{1}{2}$ is located at
$(0,1)$. The axis is then given by the equation $x=0$. Also
note that the mirror is located above the graph of the function
$y=|x|$.

A particle moving ``downwards'', i.e., along the direction
$(0,-1)$, would be reflected towards the focal point due to the
reflective property of a parabola (see
Fig.~\ref{fig:parab01}\,(a)).
\begin{figure}[h]
\centering
\begin{picture}(420,206)
\put(19,30){\includegraphics[height=176pt,
keepaspectratio]{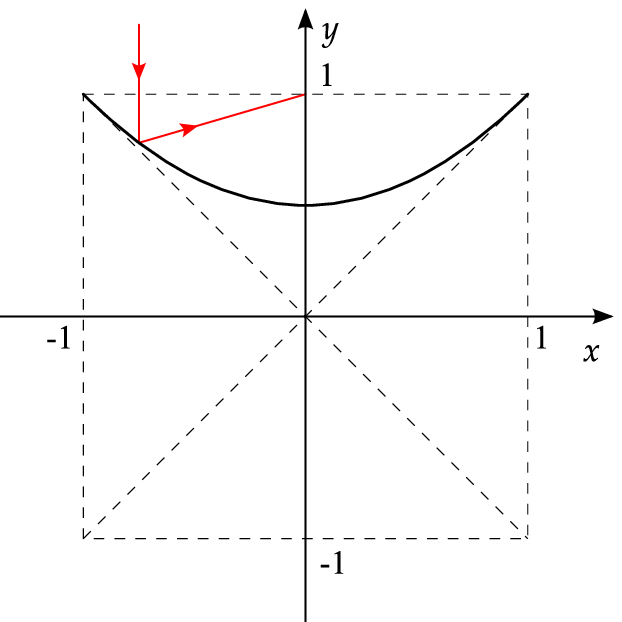}}
\put(225,30){\includegraphics[height=176pt,
keepaspectratio]{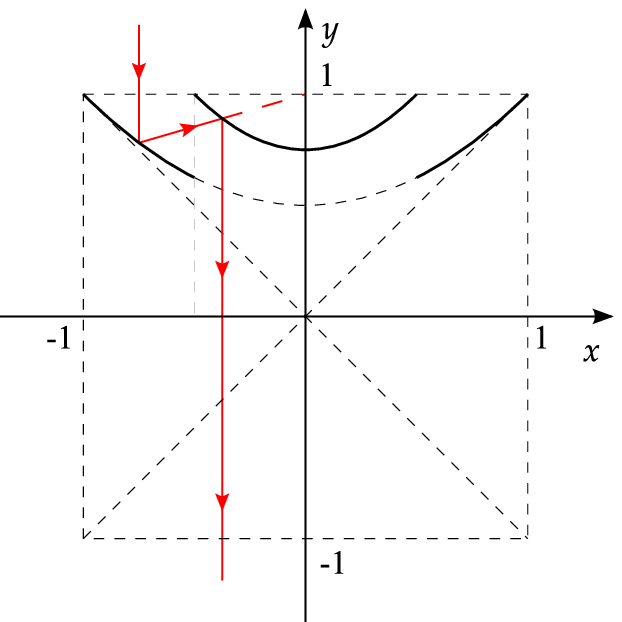}}
\put(19,10){(a)}
\put(225,10){(b)}
\end{picture}
\caption{Fractal body invisible in two directions: (a) constructing the first parabolic mirror; (b) adding a similar confocal mirror.}
\label{fig:parab01}
\end{figure}

We add one more parabolic mirror $p_2$ to our construction.
This mirror is similar to the previous one, only two times
smaller, while the foci of the corresponding parabolas coincide
at $(0,1)$. We let
$$
p_2 = \{(x,y)\, :\, y = x^2+\frac{3}{4},\ |x|\leq \frac{1}{2}\}.
$$

Now all particles that go in the downward direction and pass
the line segment $[-1,-\frac{1}{2}]\times\{1\})$ (as well as
$[\frac{1}{2},1]\times\{1\})$), are first reflected towards the
focal point, move towards $p_2$, and after the collision with
$p_2$ are redirected downwards (because of the aforementioned
property of a parabola and coinciding foci of the two
parabolas). If we remove a piece that is directly behind $p_2$
from $p_1$, the resulting body
$$
p'_1 = \{(x,y)\,:\, y = \frac{1}{2}x^2+\frac{1}{2},\ \frac{1}{2}\leq |x|\leq 1\},
$$
is not obstructing the further movement of the particles, and
they leave the body with the same velocity as they had before
entering the body (see Fig.~\ref{fig:parab01}\,(b)).

If we repeat this construction process infinitely, adding
figures similar to $p_1$ and cutting out the middle sections
of the relevant parabolas, we obtain a sequence $\{p'_k\}_1^\infty$ of parabolic mirror segments:
$$
p_k' = \left\{(x,y)\,:\, y = 2^{k-2}x^2+1-2^{-k}, \; 2^{-k}\leq |x|\leq 2^{-k+1}\right\}.
$$
This mirror sequence is plotted in Fig.~\ref{fig:parab02}\,(a). The union of these segments is denoted by $P = \cup_{k=1}^\infty p'_k$.

Take the sequence of parabolic mirror pieces $ q_k$ symmetric to $p'_k$ with respect to the $x$-axis,
$$
q_k = \left\{(x,y)\,:\, y = -2^{k-2}x^2-1+2^{-k}, \; 2^{-k}\leq |x|\leq 2^{-k+1}\right\}
$$
and let $Q = \cup_{k=1}^\infty q_k$; the union $P \cup Q$ is a fractal body invisible in the direction $(0,\pm 1)$ (see Fig.~\ref{fig:parab02}\,(b)).
\begin{figure}[h]
\centering
\begin{picture}(420,206)
\put(19,30){\includegraphics[height=176pt,
keepaspectratio]{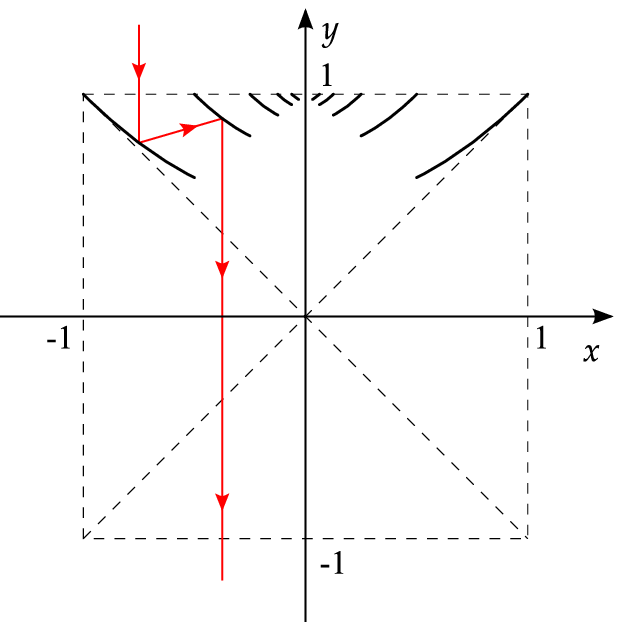}}
\put(225,30){\includegraphics[height=176pt,
keepaspectratio]{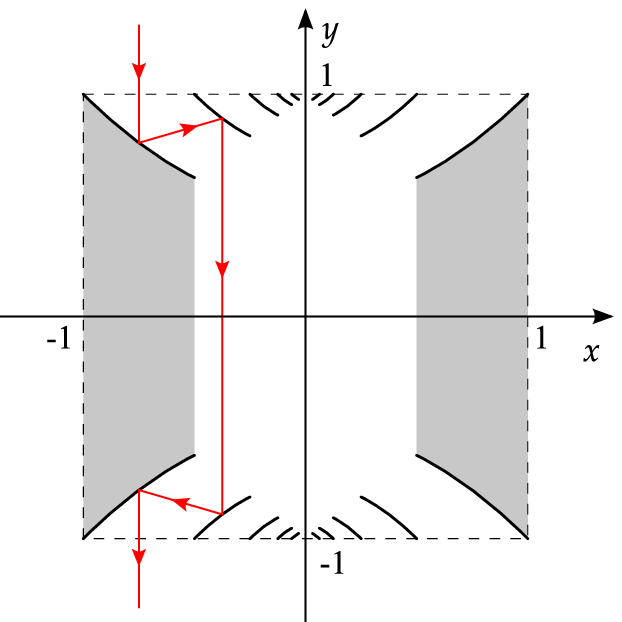}}
\put(19,10){(a)}
\put(225,10){(b)}
\end{picture}
\caption{Fractal body invisible in two directions: (a)
the basic fractal construction; (b) body invisible in the vertical direction.}
\label{fig:parab02}
\end{figure}
Indeed, since the lower part $Q$ of our body is symmetric to
the upper part $P$, it is redirecting the particles back to
their original trajectories.

Observe that the area
$$
G = \{(x,y)\,:\, |y|\leq \frac{1}{2}x^2+\frac{1}{2}, \ \frac{1}{2}\leq |x|\leq 1\}
$$
(greyed in Fig.~\ref{fig:parab02}\,(b)) is completely
``shaded'' from the particles moving parallel to the $y$-axis.
We can hence use this area to make our body invisible in a
second direction. We simply add one more construction identical
to the original one, but rotated by $\frac{\pi}{2}$. It is not
difficult to observe that it fits into $G$, and also makes the
body invisible in the direction $(\pm 1,0)$. Also observe that
the four grey blocks in the corners of the square (see
Fig.~\ref{fig:parab03}), are never accessed by the particles
moving in the directions of invisibility. We can hence include
these areas into our body.
\begin{figure}[h]
\centering
\begin{picture}(176,176)
\put(0,0){\includegraphics[height=176pt,
keepaspectratio]{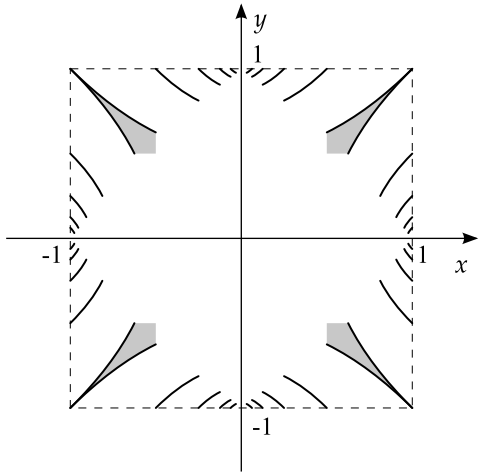}}
\end{picture}
\caption{A thin fractal body invisible in two orthogonal directions}\label{fig:parab03}
\end{figure}

Thus, we have proved the following

\begin{theorem}\label{thm:2Dthin}
There exists a {\rm thin fractal} body invisible in two perpendicular directions.
\end{theorem}

\subsection{A body invisible in two arbitrary directions}\label{ss:2dir nonort}

In this section we go further and construct a two-dimensional fractal body (without thin parts) invisible in {\em any} two directions.

This time we construct our body inside a rhombus $ABCD$ with sides parallel to the directions of invisibility.
Assume that we are given two non-parallel directions, not necessarily perpendicular. 
We rotate the coordinate system in a way that one of the
directions is vertical. Our rhombus has got two sides parallel
to the $y$-axis, and the other two parallel to the second
direction of invisibility. The centre of the rhombus coincides
with the origin.

This time our construction requires some preparatory work. Denote by $-c$ the abscissa of $A$ and select two infinite sequences of positive values $a_0,\, a_1,\, a_2, \ldots$ and $c_0,\, c_1,\, c_2, \ldots$ recursively according to the following rules.

Let $c_0 = c$ and $a_0 = +\infty$, and choose arbitrary $c_1$ satisfying $0 < c_1 < c_0$. On the $i$th step of the procedure, $i = 1,\, 2, \ldots$ select arbitrary $a_i$ satisfying the inequalities
\beq\label{ai}
0 < a_i < a_{i-1} \quad \text{and} \quad a_i(c_{i-1} -2c_i) < c_i^2
\eeq
and put
\beq\label{ci}
c_{i+1} = \frac{(c_i + a_i)^2}{c_{i-1} + a_i} - a_i.
\eeq
Using (\ref{ai}) and (\ref{ci}), one easily gets that $c_{i} > 0$ and derives by induction that $c_{i+1} < c_i$.

Now we are ready to draw the parabolas of our construction. In the description below, the points on the side $AB$ are marked by the values of their abscissas: $-\mathfrak{c}_0,\, -\mathfrak{c}_1, -\mathfrak{c}_2,\ldots, \mathfrak{a}_1,\, \mathfrak{a}_2,\ldots$ (see Fig. \ref{fig rhomb prelim}).

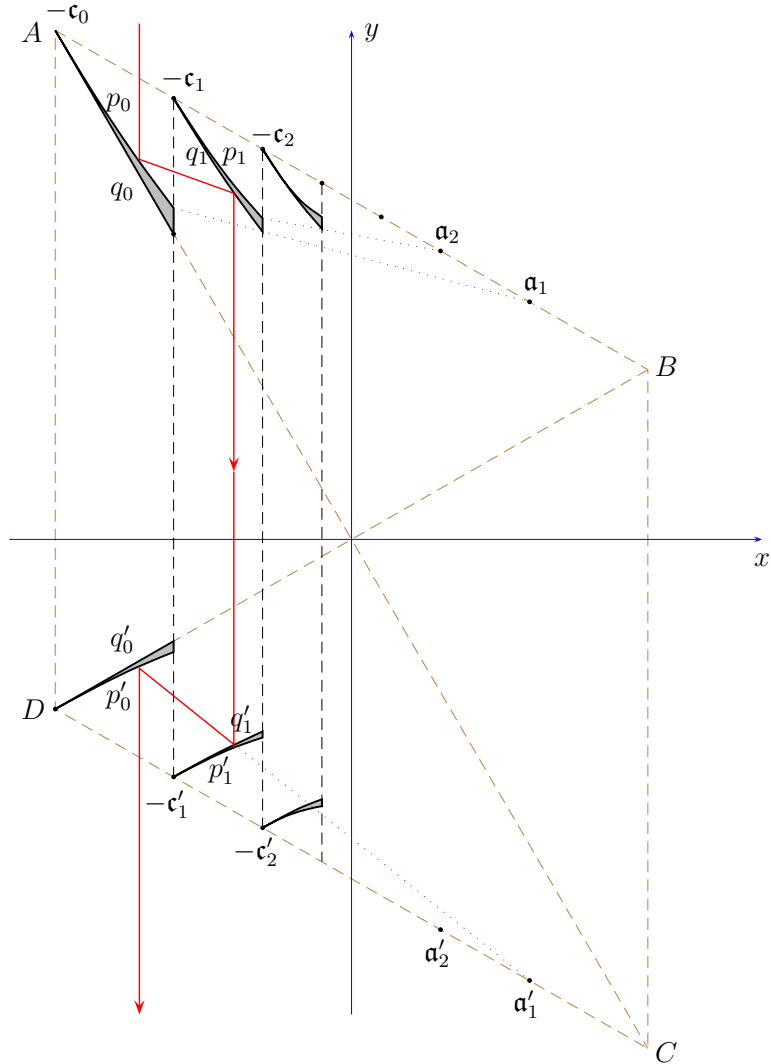
\begin{figure}
\begin{picture}(0,385)
\rput(7.7,11.3){
\scalebox{0.9}{
\pspolygon[linestyle=dashed,linewidth=0.4pt,linecolor=brown](-4.33,2.5)(4.33,-2.5)(4.33,-12.5)(-4.33,-7.5)
\psline[linestyle=dashed,linewidth=0.4pt,linecolor=brown](4.33,-2.5)(-4.33,-7.5)
\psline[linestyle=dashed,linewidth=0.4pt,linecolor=brown](-4.33,2.5)(4.33,-12.5)
\psline[linewidth=0.3pt,linecolor=blue,arrows=->,arrowscale=1.2](-5,-5)(6,-5)
\rput(6,-5.3){$x$}
\psline[linewidth=0.3pt,linecolor=blue,arrows=->,arrowscale=1.2](0,-12)(0,2.5)
\rput(0.3,2.47){$y$}
\psdots[dotsize=2pt](-2.6,1.5)(-1.3,0.75)(-0.433,0.25)(2.6,-1.5)(1.3,-0.75)(0.433,-0.25)
(-2.6,-0.5)(2.6,-11.5)(-4.33,-7.5)(-2.6,-8.5)(1.3,-10.75)(-1.3,-9.25)
\psdots[dotsize=1pt](-1.3,0.75)(-0.9,0.15)(-0.65,-0.17)(-0.433,-0.42)
  (-1.3,-9.25)(-1,-9.095)(-0.7,-8.99)(-0.433,-8.93)
   \pscustom[linewidth=0.8pt,fillstyle=solid,fillcolor=lightgray]{
\pscurve[linewidth=0.3pt](-4.33,2.5)(-3.75,1.54)(-3.18,0.678)(-2.6,-0.118)
\psline[linewidth=0.3pt,liftpen=1](-2.6,-0.5)(-4.33,2.5)}
   \pscustom[linewidth=0.8pt,fillstyle=solid,fillcolor=lightgray]{
\pscurve[linewidth=0.3pt](-2.6,1.495)(-2.1,0.678)(-1.7,0.08)(-1.3,-0.467)
\pscurve[linewidth=0.3pt,liftpen=1](-1.3,-0.27)(-1.7,0.196)(-2.1,0.735)(-2.6,1.5)}
   \pscustom[linewidth=0.8pt,fillstyle=solid,fillcolor=lightgray]{
\pscurve[linewidth=0.3pt](-1.3,0.75)(-0.9,0.15)(-0.65,-0.17)(-0.433,-0.422)
\pscurve[linewidth=0.3pt,liftpen=1](-0.433,-0.249)(-0.7,-0.05)(-1,0.295)(-1.3,0.75)}
\psline[linestyle=dotted,linewidth=0.4pt](-2.6,-0.118)(2.6,-1.5)
\psline[linestyle=dotted,linewidth=0.4pt](-1.3,-0.27)(1.3,-0.75)
     \pscustom[linewidth=0.8pt,fillstyle=solid,fillcolor=lightgray]{
   \psline[linewidth=0.3pt](-4.33,-7.5)(-2.6,-6.5)
   \pscurve[linewidth=0.3pt,liftpen=1](-2.6,-6.66)(-3.15,-6.89)(-3.7,-7.16)(-4.33,-7.5)}
     \pscustom[linewidth=0.8pt,fillstyle=solid,fillcolor=lightgray]{
   \pscurve[linewidth=0.3pt](-2.6,-8.5)(-2.1,-8.22)(-1.7,-8.02)(-1.3,-7.83)
   \pscurve[linewidth=0.3pt,liftpen=1](-1.3,-7.92)(-1.7,-8.06)(-2.1,-8.24)(-2.6,-8.5)}
     \pscustom[linewidth=0.8pt,fillstyle=solid,fillcolor=lightgray]{
   \pscurve[linewidth=0.3pt](-1.3,-9.25)(-1,-9.082)(-0.7,-8.933)(-0.433,-8.83)
   \pscurve[linewidth=0.3pt,liftpen=1](-0.433,-8.93)(-0.7,-8.99)(-1,-9.095)(-1.3,-9.25)}
\psline[linestyle=dashed,linewidth=0.3pt](-2.6,1.5)(-2.6,-8.5)
\psline[linestyle=dashed,linewidth=0.3pt](-1.3,0.75)(-1.3,-9.25)
\psline[linestyle=dashed,linewidth=0.3pt](-0.433,0.25)(-0.433,-9.75)
 \psline[linecolor=red,linewidth=0.6pt,arrows=->,arrowscale=1.6](-3.1,2.6)(-3.1,0.6)
    (-1.72,0.1)(-1.72,-4)  
    \psline[linecolor=red,linewidth=0.6pt,arrows=->,arrowscale=1.6](-1.72,-4)(-1.72,-8.02)(-3.1,-6.9)(-3.1,-12)
    \psline[linecolor=red,linestyle=dotted,linewidth=0.5pt](-3.1,-6.9)(2.6,-11.5)
\rput(-4.67,2.47){$A$}
\rput(4.6,-2.45){$B$}
\rput(4.6,-12.55){$C$}
\rput(-4.65,-7.5){$D$}
\rput(-4.15,2.77){$-\mathfrak{c}_0$}
\rput(-2.45,1.77){$-\mathfrak{c}_1$}
\rput(-1.15,0.95){$-\mathfrak{c}_2$}
\rput(2.7,-1.25){$\mathfrak{a}_1$}
\rput(1.4,-0.5){$\mathfrak{a}_2$}
\rput(-2.7,-8.9){$-\mathfrak{c}_1'$}
\rput(-1.4,-9.6){$-\mathfrak{c}_2'$}
\rput(2.55,-11.85){$\mathfrak{a}_1'$}
\rput(1.25,-11.1){$\mathfrak{a}_2'$}
\rput(-3.35,0.12){$q_0$}
\rput(-3.35,-6.47){$q_0'$}
\rput(-3.4,1.43){$p_0$}
\rput(-3.4,-7.3){$p_0'$}
 \rput(-2.25,0.64){$q_1$}
 \rput(-1.6,-7.65){$q_1'$}
 \rput(-1.7,0.65){$p_1$}
 \rput(-1.9,-8.4){$p_1'$}
}}
\end{picture}
\caption{Preliminary construction: A body inside a rhombus invisible in one direction.}
\label{fig rhomb prelim}
\end{figure}

For any $i = 0,\, 1,\, 2,\ldots$ consider the parabola through
$-\mathfrak{c}_i$ with a vertical axis and with the focus at
$a_{i+1}$. Denote by $p_i$ the arc of this parabola bounded by
the points with abscissas $-c_{i}$ and $-c_{i+1}$. Then for any
$i = 1,\, 2,\ldots$ consider the parabola through
$-\mathfrak{c}_i$ with a vertical axis and with the focus at
$a_{i}$. The arc of this parabola situated between the points
with abscissas $-c_{i}$ and $-c_{i+1}$ is denoted by $q_i$.
Also denote by $q_0$ the part of the diagonal $AC$ between $A$
and the point with abscissa $-c_1$. Each arc $p_i$ is situated
above $q_i$; more precisely, both $p_i$ and $q_i$ are the
graphs of functions $p_i(-t)$ and $q_i(-t)$,\, $t \in
[c_{i+1},\, c_{i}]$ with $p_i(-t) > q_i(-t)$ for any $t \in
(c_{i+1},\, c_{i}]$. Denote by $P_i,\ i = 0,\, 1,\, 2,\ldots$
the set bounded by the arcs $p_i$ and $q_i$ from above and
below and by a segment of the vertical line $x_i = -c_{i+1}$ on
the right.

Similarly, denote by $-\mathfrak{c}_i'$ and $\mathfrak{a}_i'$,
respectively, the points on the side $CD$ whose abscissas are
$-c_i$ and $a_i$ (the same values as above). Denote by $p_i'$
($q_i'$) the arc of parabola through $-\mathfrak{c}_i'$ with
vertical axis and with focus at $\mathfrak{a_{i+1}'}$
($\mathfrak{a_{i}'}$) bounded by the vertical lines $x = -c_i$
and $x = -c_{i+1}$. Introduce the sets $P_i'$ bounded by
$p_i',\ q_i'$ and the segment of line $x = -c_{i+1}$.

Denote by $H_i$\, $(i = 1,\, 2,\ldots)$ the homothety with the
centre $\mathfrak{a}_i$ and ratio $r_i = (c_i + a_i)/(c_{i-1} +
a_i)$.

\begin{proposition}\label{propoHi}
The arcs $p_{i-1}$ and $q_i$ are homothetic under $H_i$.
\end{proposition}

\begin{proof}
Note that the parabolas containing $p_{i-1}$ and $q_i$ have the
same focus at $\mathfrak{a}_i$; therefore they are homothetic
with the centre at this point. Since the points
$-\mathfrak{c}_{i-1}$ and $-\mathfrak{c}_i$ (which are the left
endpoints of the arcs $p_{i-1}$ and $q_i$) are homothetic, one
readily concludes that the ratio equals $r_i$, and therefore,
the homothety is $H_i$. Further, the right endpoint of
$p_{i-1}$ has abscissa $-c_i$; using (\ref{ci}), one verifies
that its image under $H_i$ has the abscissa $-c_{i+1}$, and
therefore, coincides with the right endpoint of $q_i$. Thus,
the arcs $p_{i-1}$ and $q_i$ belong to homothetic parabolas,
and their right and left endpoints are homothetic; therefore
they are also homothetic under $H_i$.
\end{proof}

Consider the fractal body (greyed in Figure \ref{fig rhomb
prelim})
$$
\mathcal{A}_L = (\cup_{i=0}^\infty P_i) \cup (\cup_{i=0}^\infty P_i').
$$
Now we are in a position to prove the following Proposition.

\begin{proposition}\label{invisPrelim}
The body $\mathcal{A}_L$ is invisible in the vertical direction.
\end{proposition}

\begin{proof}
Consider a particle falling vertically downward with the
velocity $(0, -1)$ along a line with abscissa $x$. We assume
that $x \in [-c,\, 0]$; otherwise the particle does not hit the
body and there is nothing to prove. If $x = -c_i \ (i = 0,\,
1,\, 2,\ldots)$, the particle hits the body at a singular point
and the motion is not defined from this point. Otherwise, $x$
belongs to an interval $(c_{i-1},\, c_i), \ i = 1,\, 2,\ldots$
($i = 1$ in Fig. \ref{fig rhomb prelim}).

The particle is first reflected by $p_{i-1}$ and then moves to the focus $\mathfrak{a}_i$; then it makes the second reflection from $q_i$ and moves along a vertical line. Since the line through the two reflection points contains the focus, these two points (and therefore the vertical lines through these points) are homothetic under $H_i$.

The particle then makes the third and fourth reflections from $q_i'$ and $p_{i-1}'$, and finally, moves freely downwards. Repeating the above argument, one concludes that the vertical lines through the third and fourth reflection points are homothetic under the inverse homothety $H_i^{-1}$, and therefore, the vertical lines containing the initial and final parts of the trajectory coincide.
\end{proof}

By adding the set $\mathcal{A}_R$ symmetric to $\mathcal{A}_L$
with respect to the centre of the rhombus, one gets the body
$\mathcal{A} = \mathcal{A}_L \cup \mathcal{A}_R$, which is also
invisible in the vertical direction $(0, -1)$ (Fig.~\ref{fig
rhomb}\,(a)).

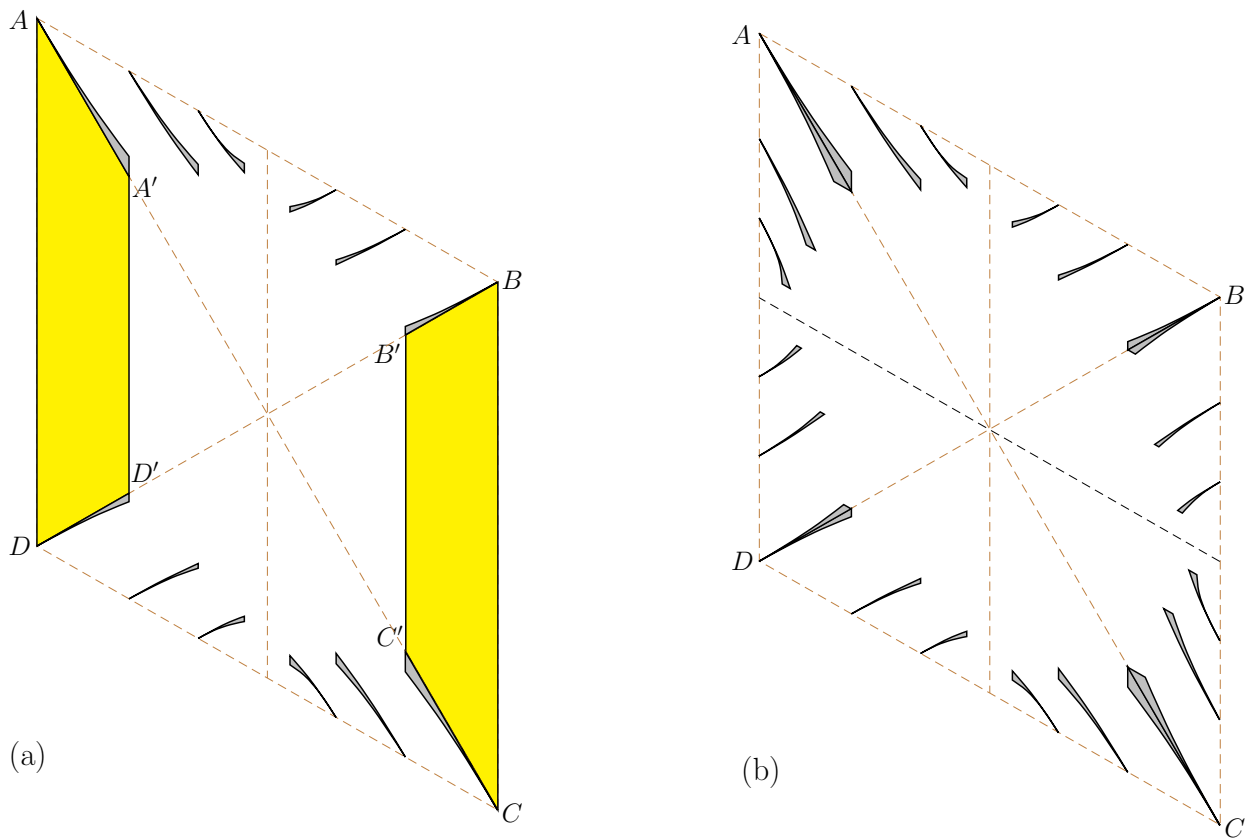
\begin{figure}
\begin{picture}(0,280)
   \rput(-2,12){
   \scalebox{0.7}{
\rput(7.5,-4){
\scalebox{1}{
\pspolygon[linestyle=dashed,linewidth=0.4pt,linecolor=brown](-4.33,2.5)(4.33,-2.5)(4.33,-12.5)(-4.33,-7.5)
   \pspolygon[fillstyle=solid,fillcolor=yellow](-4.33,2.5)(-2.6,-0.5)(-2.6,-6.5)(-4.33,-7.5)
   \pspolygon[fillstyle=solid,fillcolor=yellow](4.33,-2.5)(2.6,-3.5)(2.6,-9.5)(4.33,-12.5)
   \rput(-2.3,-0.7){\large $A'$}
   \rput(2.25,-3.85){\large $B'$}
   \rput(2.3,-9.2){\large $C'$}
   \rput(-2.3,-6.1){\large $D'$}
\psline[linestyle=dashed,linewidth=0.4pt,linecolor=brown](4.33,-2.5)(-4.33,-7.5)
\psline[linestyle=dashed,linewidth=0.4pt,linecolor=brown](-4.33,2.5)(4.33,-12.5)
\psline[linestyle=dashed,linewidth=0.4pt,linecolor=brown](0,0)(0,-10)
   \pscustom[linewidth=0.8pt,fillstyle=solid,fillcolor=lightgray]{
\pscurve[linewidth=0.3pt](-4.33,2.5)(-3.75,1.54)(-3.18,0.678)(-2.6,-0.118)
\psline[linewidth=0.3pt,liftpen=1](-2.6,-0.5)(-4.33,2.5)}
   \pscustom[linewidth=0.8pt,fillstyle=solid,fillcolor=lightgray]{
\pscurve[linewidth=0.3pt](-2.6,1.495)(-2.1,0.678)(-1.7,0.08)(-1.3,-0.467)
\pscurve[linewidth=0.3pt,liftpen=1](-1.3,-0.27)(-1.7,0.196)(-2.1,0.735)(-2.6,1.5)}
   \pscustom[linewidth=0.8pt,fillstyle=solid,fillcolor=lightgray]{
\pscurve[linewidth=0.3pt](-1.3,0.75)(-0.9,0.15)(-0.65,-0.17)(-0.433,-0.422)
\pscurve[linewidth=0.3pt,liftpen=1](-0.433,-0.249)(-0.7,-0.05)(-1,0.295)(-1.3,0.75)}
     \pscustom[linewidth=0.8pt,fillstyle=solid,fillcolor=lightgray]{
   \psline[linewidth=0.3pt](-4.33,-7.5)(-2.6,-6.5)
   \pscurve[linewidth=0.3pt,liftpen=1](-2.6,-6.66)(-3.15,-6.89)(-3.7,-7.16)(-4.33,-7.5)}
     \pscustom[linewidth=0.8pt,fillstyle=solid,fillcolor=lightgray]{
   \pscurve[linewidth=0.3pt](-2.6,-8.5)(-2.1,-8.22)(-1.7,-8.02)(-1.3,-7.83)
   \pscurve[linewidth=0.3pt,liftpen=1](-1.3,-7.92)(-1.7,-8.06)(-2.1,-8.24)(-2.6,-8.5)}
     \pscustom[linewidth=0.8pt,fillstyle=solid,fillcolor=lightgray]{
   \pscurve[linewidth=0.3pt](-1.3,-9.25)(-1,-9.082)(-0.7,-8.933)(-0.433,-8.83)
   \pscurve[linewidth=0.3pt,liftpen=1](-0.433,-8.93)(-0.7,-8.99)(-1,-9.095)(-1.3,-9.25)}
\rput(-4.67,2.47){\large $A$}
\rput(4.6,-2.45){\large $B$}
\rput(4.6,-12.55){\large $C$}
\rput(-4.65,-7.5){\large $D$}
\rput(-4.5,-11.5){\Large (a)}
}}
\rput{180}(7.77,-14){
\scalebox{1}{
   \pscustom[linewidth=0.8pt,fillstyle=solid,fillcolor=lightgray]{
\pscurve[linewidth=0.3pt](-4.33,2.5)(-3.75,1.54)(-3.18,0.678)(-2.6,-0.118)
\psline[linewidth=0.3pt,liftpen=1](-2.6,-0.5)(-4.33,2.5)}
   \pscustom[linewidth=0.8pt,fillstyle=solid,fillcolor=lightgray]{
\pscurve[linewidth=0.3pt](-2.6,1.495)(-2.1,0.678)(-1.7,0.08)(-1.3,-0.467)
\pscurve[linewidth=0.3pt,liftpen=1](-1.3,-0.27)(-1.7,0.196)(-2.1,0.735)(-2.6,1.5)}
   \pscustom[linewidth=0.8pt,fillstyle=solid,fillcolor=lightgray]{
\pscurve[linewidth=0.3pt](-1.3,0.75)(-0.9,0.15)(-0.65,-0.17)(-0.433,-0.422)
\pscurve[linewidth=0.3pt,liftpen=1](-0.433,-0.249)(-0.7,-0.05)(-1,0.295)(-1.3,0.75)}
     \pscustom[linewidth=0.8pt,fillstyle=solid,fillcolor=lightgray]{
   \psline[linewidth=0.3pt](-4.33,-7.5)(-2.6,-6.5)
   \pscurve[linewidth=0.3pt,liftpen=1](-2.6,-6.66)(-3.15,-6.89)(-3.7,-7.16)(-4.33,-7.5)}
     \pscustom[linewidth=0.8pt,fillstyle=solid,fillcolor=lightgray]{
   \pscurve[linewidth=0.3pt](-2.6,-8.5)(-2.1,-8.22)(-1.7,-8.02)(-1.3,-7.83)
   \pscurve[linewidth=0.3pt,liftpen=1](-1.3,-7.92)(-1.7,-8.06)(-2.1,-8.24)(-2.6,-8.5)}
     \pscustom[linewidth=0.8pt,fillstyle=solid,fillcolor=lightgray]{
   \pscurve[linewidth=0.3pt](-1.3,-9.25)(-1,-9.082)(-0.7,-8.933)(-0.433,-8.83)
   \pscurve[linewidth=0.3pt,liftpen=1](-0.433,-8.93)(-0.7,-8.99)(-1,-9.095)(-1.3,-9.25)}
}} }}
   \rput(7.5,11.8){
   \scalebox{0.7}{
\rput(7.5,-4){
\scalebox{1}{
\pspolygon[linestyle=dashed,linewidth=0.4pt,linecolor=brown](-4.33,2.5)(4.33,-2.5)(4.33,-12.5)(-4.33,-7.5)
\psline[linestyle=dashed,linewidth=0.4pt,linecolor=brown](4.33,-2.5)(-4.33,-7.5)
\psline[linestyle=dashed,linewidth=0.4pt,linecolor=brown](-4.33,2.5)(4.33,-12.5)
\psline[linestyle=dashed,linewidth=0.4pt,linecolor=brown](0,0)(0,-10)
   \pscustom[linewidth=0.8pt,fillstyle=solid,fillcolor=lightgray]{
\pscurve[linewidth=0.3pt](-4.33,2.5)(-3.75,1.54)(-3.18,0.678)(-2.6,-0.118)
\psline[linewidth=0.3pt,liftpen=1](-2.6,-0.5)(-4.33,2.5)}
   \pscustom[linewidth=0.8pt,fillstyle=solid,fillcolor=lightgray]{
\pscurve[linewidth=0.3pt](-2.6,1.495)(-2.1,0.678)(-1.7,0.08)(-1.3,-0.467)
\pscurve[linewidth=0.3pt,liftpen=1](-1.3,-0.27)(-1.7,0.196)(-2.1,0.735)(-2.6,1.5)}
   \pscustom[linewidth=0.8pt,fillstyle=solid,fillcolor=lightgray]{
\pscurve[linewidth=0.3pt](-1.3,0.75)(-0.9,0.15)(-0.65,-0.17)(-0.433,-0.422)
\pscurve[linewidth=0.3pt,liftpen=1](-0.433,-0.249)(-0.7,-0.05)(-1,0.295)(-1.3,0.75)}
     \pscustom[linewidth=0.8pt,fillstyle=solid,fillcolor=lightgray]{
   \psline[linewidth=0.3pt](-4.33,-7.5)(-2.6,-6.5)
   \pscurve[linewidth=0.3pt,liftpen=1](-2.6,-6.66)(-3.15,-6.89)(-3.7,-7.16)(-4.33,-7.5)}
     \pscustom[linewidth=0.8pt,fillstyle=solid,fillcolor=lightgray]{
   \pscurve[linewidth=0.3pt](-2.6,-8.5)(-2.1,-8.22)(-1.7,-8.02)(-1.3,-7.83)
   \pscurve[linewidth=0.3pt,liftpen=1](-1.3,-7.92)(-1.7,-8.06)(-2.1,-8.24)(-2.6,-8.5)}
     \pscustom[linewidth=0.8pt,fillstyle=solid,fillcolor=lightgray]{
   \pscurve[linewidth=0.3pt](-1.3,-9.25)(-1,-9.082)(-0.7,-8.933)(-0.433,-8.83)
   \pscurve[linewidth=0.3pt,liftpen=1](-0.433,-8.93)(-0.7,-8.99)(-1,-9.095)(-1.3,-9.25)}
\rput(-4.67,2.47){\large $A$}
\rput(4.6,-2.45){\large $B$}
\rput(4.6,-12.55){\large $C$}
\rput(-4.65,-7.5){\large $D$}
\rput(-4.3,-11.5){\Large (b)}
}}
\rput{180}(7.77,-14){
\scalebox{1}{
   \pscustom[linewidth=0.8pt,fillstyle=solid,fillcolor=lightgray]{
\pscurve[linewidth=0.3pt](-4.33,2.5)(-3.75,1.54)(-3.18,0.678)(-2.6,-0.118)
\psline[linewidth=0.3pt,liftpen=1](-2.6,-0.5)(-4.33,2.5)}
   \pscustom[linewidth=0.8pt,fillstyle=solid,fillcolor=lightgray]{
\pscurve[linewidth=0.3pt](-2.6,1.495)(-2.1,0.678)(-1.7,0.08)(-1.3,-0.467)
\pscurve[linewidth=0.3pt,liftpen=1](-1.3,-0.27)(-1.7,0.196)(-2.1,0.735)(-2.6,1.5)}
   \pscustom[linewidth=0.8pt,fillstyle=solid,fillcolor=lightgray]{
\pscurve[linewidth=0.3pt](-1.3,0.75)(-0.9,0.15)(-0.65,-0.17)(-0.433,-0.422)
\pscurve[linewidth=0.3pt,liftpen=1](-0.433,-0.249)(-0.7,-0.05)(-1,0.295)(-1.3,0.75)}
     \pscustom[linewidth=0.8pt,fillstyle=solid,fillcolor=lightgray]{
   \psline[linewidth=0.3pt](-4.33,-7.5)(-2.6,-6.5)
   \pscurve[linewidth=0.3pt,liftpen=1](-2.6,-6.66)(-3.15,-6.89)(-3.7,-7.16)(-4.33,-7.5)}
     \pscustom[linewidth=0.8pt,fillstyle=solid,fillcolor=lightgray]{
   \pscurve[linewidth=0.3pt](-2.6,-8.5)(-2.1,-8.22)(-1.7,-8.02)(-1.3,-7.83)
   \pscurve[linewidth=0.3pt,liftpen=1](-1.3,-7.92)(-1.7,-8.06)(-2.1,-8.24)(-2.6,-8.5)}
     \pscustom[linewidth=0.8pt,fillstyle=solid,fillcolor=lightgray]{
   \pscurve[linewidth=0.3pt](-1.3,-9.25)(-1,-9.082)(-0.7,-8.933)(-0.433,-8.83)
   \pscurve[linewidth=0.3pt,liftpen=1](-0.433,-8.93)(-0.7,-8.99)(-1,-9.095)(-1.3,-9.25)}
}}
 \rput{-120}(19.28,-6.83){
\rput(7.5,-4){
\scalebox{1}{
\psline[linestyle=dashed,linewidth=0.4pt](0,0)(0,-10)
   \pscustom[linewidth=0.8pt,fillstyle=solid,fillcolor=lightgray]{
\pscurve[linewidth=0.3pt](4.33,2.5)(3.75,1.54)(3.18,0.678)(2.6,-0.118)
\psline[linewidth=0.3pt,liftpen=1](2.6,-0.5)(4.33,2.5)}
   \pscustom[linewidth=0.8pt,fillstyle=solid,fillcolor=lightgray]{
\pscurve[linewidth=0.3pt](2.6,1.495)(2.1,0.678)(1.7,0.08)(1.3,-0.467)
\pscurve[linewidth=0.3pt,liftpen=1](1.3,-0.27)(1.7,0.196)(2.1,0.735)(2.6,1.5)}
   \pscustom[linewidth=0.8pt,fillstyle=solid,fillcolor=lightgray]{
\pscurve[linewidth=0.3pt](1.3,0.75)(0.9,0.15)(0.65,-0.17)(0.433,-0.422)
\pscurve[linewidth=0.3pt,liftpen=1](0.433,-0.249)(0.7,-0.05)(1,0.295)(1.3,0.75)}
     \pscustom[linewidth=0.8pt,fillstyle=solid,fillcolor=lightgray]{
   \psline[linewidth=0.3pt](4.33,-7.5)(2.6,-6.5)
   \pscurve[linewidth=0.3pt,liftpen=1](2.6,-6.66)(3.15,-6.89)(3.7,-7.16)(4.33,-7.5)}
     \pscustom[linewidth=0.8pt,fillstyle=solid,fillcolor=lightgray]{
   \pscurve[linewidth=0.3pt](2.6,-8.5)(2.1,-8.22)(1.7,-8.02)(1.3,-7.83)
   \pscurve[linewidth=0.3pt,liftpen=1](1.3,-7.92)(1.7,-8.06)(2.1,-8.24)(2.6,-8.5)}
     \pscustom[linewidth=0.8pt,fillstyle=solid,fillcolor=lightgray]{
   \pscurve[linewidth=0.3pt](1.3,-9.25)(1,-9.082)(0.7,-8.933)(0.433,-8.83)
   \pscurve[linewidth=0.3pt,liftpen=1](0.433,-8.93)(0.7,-8.99)(1,-9.095)(1.3,-9.25)}
}}
\rput{180}(7.77,-14){
\scalebox{1}{
   \pscustom[linewidth=0.8pt,fillstyle=solid,fillcolor=lightgray]{
\pscurve[linewidth=0.3pt](4.33,2.5)(3.75,1.54)(3.18,0.678)(2.6,-0.118)
\psline[linewidth=0.3pt,liftpen=1](2.6,-0.5)(4.33,2.5)}
   \pscustom[linewidth=0.8pt,fillstyle=solid,fillcolor=lightgray]{
\pscurve[linewidth=0.3pt](2.6,1.495)(2.1,0.678)(1.7,0.08)(1.3,-0.467)
\pscurve[linewidth=0.3pt,liftpen=1](1.3,-0.27)(1.7,0.196)(2.1,0.735)(2.6,1.5)}
   \pscustom[linewidth=0.8pt,fillstyle=solid,fillcolor=lightgray]{
\pscurve[linewidth=0.3pt](1.3,0.75)(0.9,0.15)(0.65,-0.17)(0.433,-0.422)
\pscurve[linewidth=0.3pt,liftpen=1](0.433,-0.249)(0.7,-0.05)(1,0.295)(1.3,0.75)}
     \pscustom[linewidth=0.8pt,fillstyle=solid,fillcolor=lightgray]{
   \psline[linewidth=0.3pt](4.33,-7.5)(2.6,-6.5)
   \pscurve[linewidth=0.3pt,liftpen=1](2.6,-6.66)(3.15,-6.89)(3.7,-7.16)(4.33,-7.5)}
     \pscustom[linewidth=0.8pt,fillstyle=solid,fillcolor=lightgray]{
   \pscurve[linewidth=0.3pt](2.6,-8.5)(2.1,-8.22)(1.7,-8.02)(1.3,-7.83)
   \pscurve[linewidth=0.3pt,liftpen=1](1.3,-7.92)(1.7,-8.06)(2.1,-8.24)(2.6,-8.5)}
     \pscustom[linewidth=0.8pt,fillstyle=solid,fillcolor=lightgray]{
   \pscurve[linewidth=0.3pt](1.3,-9.25)(1,-9.082)(0.7,-8.933)(0.433,-8.83)
   \pscurve[linewidth=0.3pt,liftpen=1](0.433,-8.93)(0.7,-8.99)(1,-9.095)(1.3,-9.25)}
}}
}}}

\end{picture}
\caption{(a) A centrally symmetric body invisible in one direction and the shaded regions. (b) A body invisible in two directions.}
\label{fig rhomb}
\end{figure}

\begin{remark}\label{rem:1}
The arcs $p_i,\, q_i,\, p_i',\, q_i',\ i = 0,\, 1,\, 2,\ldots$ are graphs of functions; denote these functions by $p_i(-x)$,\, $q_i(-x)$,\, $p{\text{\!'}}_{\!i}(-x)$,\, $q{\text{\!'}}_{\!i}(-x)$,\, $x \in [c_{i+1},\, c_i]$. Then the sets $P_i,\, P_i'$ can be represented as
$$
P_i = \{ (x,y):\, -c_{i} \le x \le -c_{i+1},\ q_i(-x) \le y \le p_i(-x) \},
$$
$$
P_i' = \{ (x,y):\, -c_{i} \le x \le -c_{i+1},\ p{\text{\!'}}_{\!i}(-x) \le y \le q{\text{\!'}}_{\!i}(-x) \}.
$$
In particular, if the directions of invisibility are orthogonal, we have $p{\text{\!'}}_{\!i}(x) = -p_i(x)$,\, $q{\text{\!'}}_{\!i}(x) = -q_i(x)$, and the body $\AAA$ is analytically described as
$$
\AAA = \cup_{i=0}^\infty \{ (x,y):\, c_{i+1} \le |x| \le c_i, \ q_i(|x|) \le |y| \le p_i(|x|) \}.
$$
This representation will be used in the next section.
\end{remark}

Observe that the trapezoids $AA'D'D$ and $BB'C'C$ bounded by
the lines $x = \pm c_0$,\, $x = \pm c_1$ and by the diagonals
of the rhombus are completely ``shaded'' from the particles
falling in the vertical direction. We can therefore use this
area to make our body invisible in the second direction.

Namely, the body $\mathcal{B}$ symmetric to $\mathcal{A}$ with respect to the line $BD$ (or $AC$) is contained in the union of the trapezoids and is invisible in the direction $\overrightarrow{AB}$. The union $\mathcal{A} \cup \mathcal{B}$ is then invisible in both directions (Fig.~\ref{fig rhomb}\,(b)).

We have proved the following result. Observe that it includes Theorem \ref{thm:2Dthin} as a particular case.

\begin{theorem}\label{thm:2D}
For any two directions $v_1$ and $v_2 \in S^1$, there exists a {\rm solid fractal} body invisible in these directions.
\end{theorem}

\begin{remark}\label{rem:2}
In the limiting case $c_i = 2^{-i}c$,\, $a_i = 0\, (i \ge 1)$ one gets a {\rm thin} fractal. In particular, if the two directions are orthogonal, we have the construction described in the previous Subsection \ref{ss:2dir prelim}.
\end{remark}

\section{A body invisible in 3 directions}\label{s:3dir}

Using the two-dimensional construction described in the previous section, we can now describe a three-dimensional body invisible in 3 orthogonal
directions.

In the 3D case the construction is more complicated and intuition is less reliable; therefore we provide here a more detailed argument than in the previous section. Some of the accompanying figures, for better visibility, correspond to the limiting case of {\it thin} fractal body.

Let us first introduce some notation. We consider Euclidean
space $\RRR^3$ with orthogonal coordinates $x,\, y,\, z$ and
the cube $Q = [-c,\, c]^3$ centered at the origin $O =
(0,0,0)$. The pyramids with the vertex at $O$ and with the
bases $z = \pm c,\, |x| \le c,\, |y| \le c$ (the upper and
lower faces of the cube $Q$) are denoted by $\Pi^\pm_z$. In
other words,
$$
\Pi^\pm_z = \{ (x,y,z) \in Q :\, \pm z \ge \max\{ |x|,|y| \}\, \}.
$$
Further, $\Pi_z = \Pi^+_z \cup \Pi^-_z$. The pyramids $\Pi_x$
and $\Pi_y$ are defined in the same way. Notice that the
interiors of $\Pi_x,\, \Pi_y$, and $\Pi_z$ are mutually
disjoint.

Let $0 < c_1 < c$. For each $\ve_x \in \{ -1,1 \}$ and $\ve_z \in \{ -1,1 \}$ we define the gallery
$$
G(\ve_x,\ve_z) = \ve_x [c_1,\, c] \times [-c,\, c] \times \ve_z [c_1,\, c]
$$
(where by definition $\ve[a,\, b] = \left\{ \begin{array}{lc}
[\ve a,\, \ve b], & \text{if } \ \ve > 0\\ {}[-\ve b,\, -\ve a]{}, & \text{if } \ \ve < 0
\end{array} \right.$\,\big); it is a horizontal parallelepiped adjacent to an edge of $Q$ parallel to the $y$-axis. The union of 4 such parallelepipeds is
$$
G_y = \cup_{\ve_x,\ve_z=\pm 1} G(\ve_x,\ve_z).
$$
The sets $G_x$ and $G_z$ are defined similarly.

We are going to define three bodies $B_{x}$,\, $B_{y}$,\, $B_{z}$, invisible in the directions $x,\, y,\, z$, respectively,  and then take their union.

First we take the two-dimensional body $\AAA = \AAA_{yz}$ in the $yz$-plane, as described in the previous subsection \ref{ss:2dir nonort}. It corresponds to 2 {\it orthogonal} directions (parallel to the $y$- and $z$-axes) and is inscribed in the square $[-c,\, c]^2$. (Notice that the body $\AAA$ shown in Fig. \ref{fig rhomb}\,(a) corresponds to 2 {\it non-orthogonal} direction and is inscribed in a rhombus.) The body $\AAA_{yz}$ invisible in the $z$-direction is determined by the parameters $c_0 = c,\, c_1,\, c_2,\ldots;\, a_1,\, a_2,\ldots$ satisfying (\ref{ai}) and (\ref{ci}).

Take the direct product
$$
\tilde\AAA_{yz} = ([-c, -c_1] \cup [c_1,\, c]) \times \AAA_{yz};
$$
the resulting three-dimensional body $\tilde\AAA_{yz}$ is also
invisible in the $z$-direction. Notice that it is contained in
the union of $y$-galleries, \beq\label{inclusion1}
\tilde\AAA_{yz} \subset G_y. \eeq The body $\tilde\AAA_{yz}$ is
shown in { Fig. \ref{fig:3d01}\,(a)}; for the sake of better
visualization we chose to draw the limiting case of ``thin
fractal'' with $a_i = 0$ and $c_i = 2^{-i}c \ (i = 1,\,
2,\ldots)$.

Let $B_{yz} = \tilde\AAA_{yz} \cap \Pi_z$ (Fig.
\ref{fig:3d01}\,(b)). Then we analogously define the body
$B_{xz}$ and take
$$
B_z = B_{yz} \cup B_{xz}.
$$
The bodies $B_x$ and $B_y$ are defined in a similar way, and finally,
$$
B = B_x \cup B_y \cup B_z.
$$
The bodies $B_z$, $B_y$ and $B_x$ are shown on
Fig.~\ref{fig:3d04} (a), Fig.~\ref{fig:3d04} (b) and
Fig.~\ref{fig:3d05} (a) respectively. The body $B$ is shown on
Fig.~\ref{fig:3d05} (b). All the pictures correspond to the
``thin'' fractal case.
\begin{figure}[h]
\centering
\begin{picture}(440,220)
\put(0,30){\includegraphics[height=190pt,
keepaspectratio]{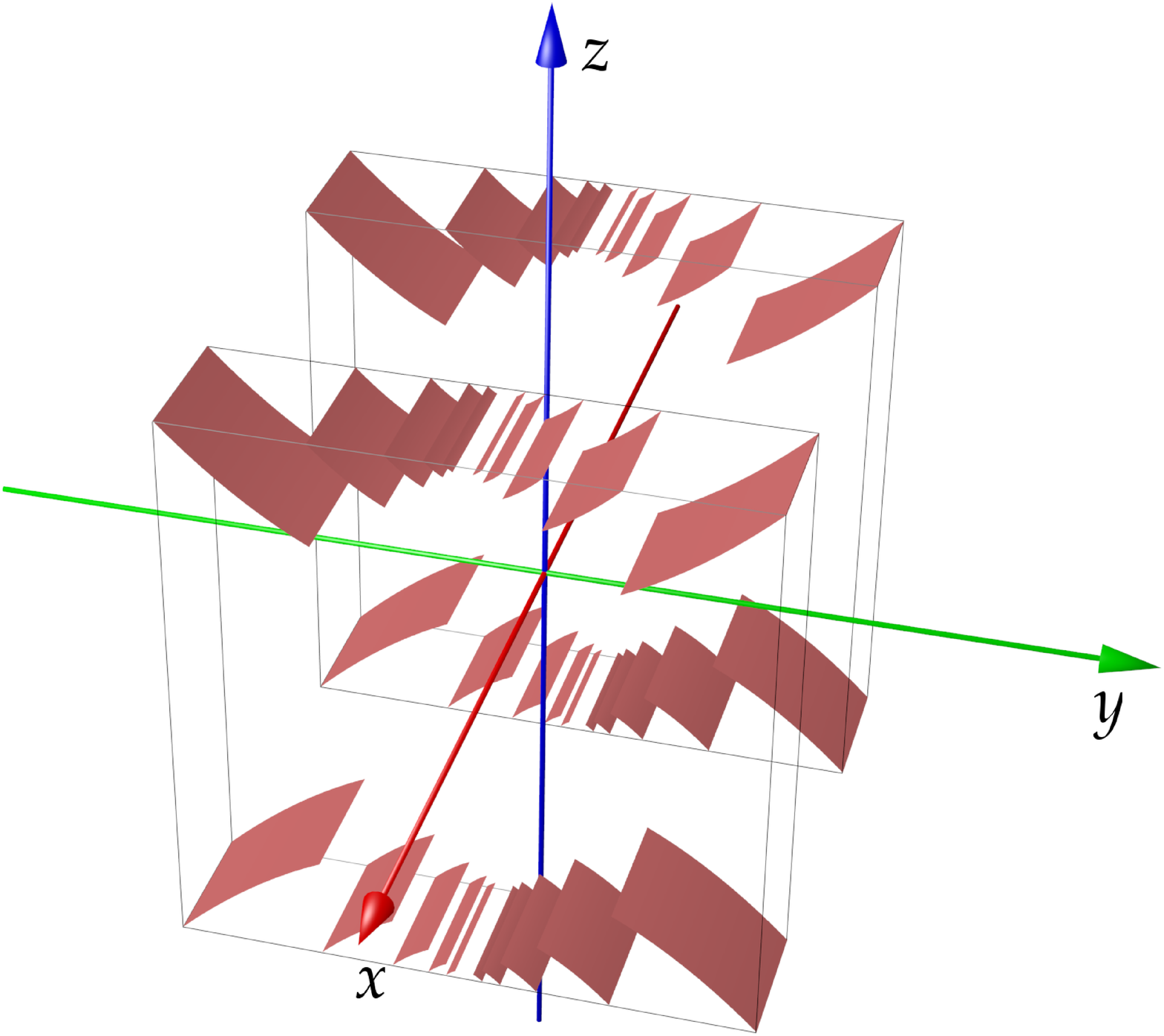}}
\put(220,30){\includegraphics[height=190pt,
keepaspectratio]{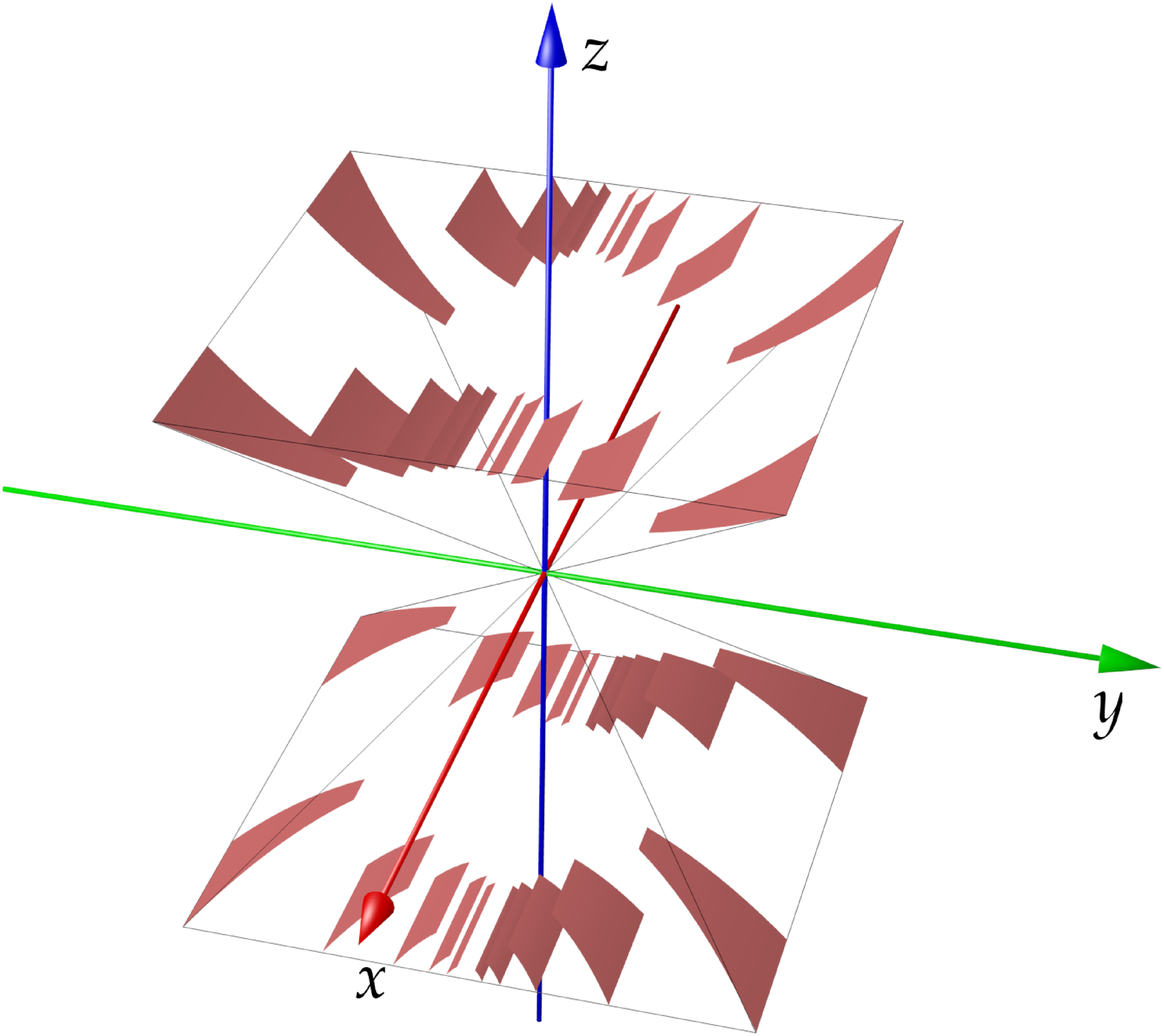}}
\put(60,10){(a)}
\put(260,10){(b)}
\end{picture}
\caption{The bodies (a) $\tilde\AAA_{yz}$ and (b) $B_{yz}$ are shown in the {\it thin fractal} case.}\label{fig:3d01}
\end{figure}

\begin{remark}\label{rem:3}
In terms of the functions $p_i$ and $q_i$ introduced in Remark \ref{rem:1}, the set $\tilde\AAA_{yz}$ can be written as
\begin{equation*}\label{til}
\tilde\AAA_{yz} = \cup_{i=0}^\infty \{ (x,y,z):\, c_{i+1} \le |y| \le c_i, \ q_i(|y|) \le |z| \le p_i(|y|), \ c_1 \le |x| \le c \}.
\end{equation*}
Taking the intersection of $\tilde\AAA_{yz}$ with $\Pi_z = \{ (x,y,z):\ |z| \ge |x|, \ |z| \ge |y| \}$ and using that $q_i(|y|) \ge |y|$, one gets
\begin{equation*}\label{byz}
B_{yz} = \cup_{i=0}^\infty \{ (x,y,z):\, c_{i+1} \le |y| \le c_i, \ q_i(|y|) \le |z| \le p_i(|y|), \ c_1 \le |x| \le |z| \},
\end{equation*}
and therefore,
\begin{equation}\label{byzincl}
B_{yz} \subset  \cup_{i=0}^\infty \{ (x,y,z):\, c_{i+1} \le |y| \le c_i, \ |x| \le p_i(|y|) \}.
\end{equation}
Similar relations are true for the other sets $B_{yx},\, B_{xz}$, etc; for example,
\begin{equation}\label{bxz}
B_{xz} = \cup_{i=0}^\infty \{ (x,y,z):\, c_{i+1} \le |x| \le c_i, \ q_i(|x|) \le |z| \le p_i(|x|), \ c_1 \le |y| \le |z| \le c \},
\end{equation}
\begin{equation}\label{byxincl}
B_{yx} \subset \cup_{i=0}^\infty \{ (x,y,z):\, c_{i+1} \le |y| \le c_i, \ q_i(|y|) \le |x| \le p_i(|y|) \}.
\end{equation}
\end{remark}

Now we are prepared for the proof of the following theorem.

\begin{theorem}\label{thm:3D}
There exists a {\rm solid fractal} body invisible in 3 mutually orthogonal directions.
\end{theorem}

\begin{proof}
We are going to show that $B$ is invisible in the directions parallel to the coordinate axes.

It suffices to prove that $B$ is invisible in the $z$-direction; this will imply invisibility in the $x$- and $y$-directions, due to symmetry of the construction under the exchange of variables $x$,\, $y$ and $z$. We will actually show that the body $B_z$ is invisible for the incident flow in the $z$-direction and the bodies $B_x$ and $B_y$ are shadowed from this flow by $B_z$.

Consider an incident particle with the velocity $(0,0,-1)$. Our goal is to prove that it makes either 0 or 4 reflections from $B_z$ (and no reflections from $B_x$ and $B_y$) and moves freely afterwards; moreover, the initial and final parts of its trajectory belong to the same straight line.

If the orthogonal projection of the coordinate of the incident particle on the $xy$-plane lies inside the square $[-c_1,\, c_1]^2$ or outside the square $[-c,\, c]^2$ (see Fig. \ref{fig:frame}), the particle does not hit the body and there is nothing to prove. It remains therefore to consider the case where the projection is contained in $[-c,\, c]^2 \setminus [-c_1,\, c_1]^2$.

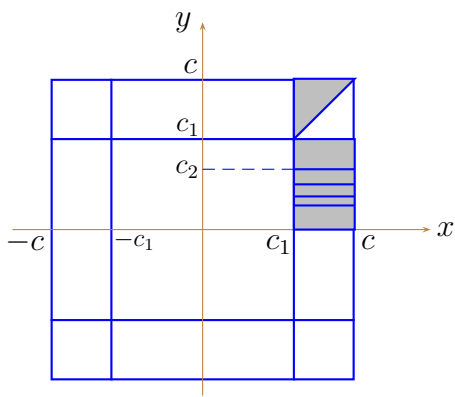
\begin{figure}
\begin{picture}(0,110)
   \rput(8,2.2){
   \psframe[linecolor=blue](-2,-2)(2,2)
   \psline[linecolor=blue](1.2,-2)(1.2,2)
   \psline[linecolor=blue](-1.2,-2)(-1.2,2)
   \psline[linecolor=blue](-2,1.2)(2,1.2)
   \psline[linecolor=blue](-2,-1.2)(2,-1.2)
      \psline[linewidth=0.4pt,linecolor=brown]{->}(-2.5,0)(3,0)
      \rput(3.2,0){$x$}
      \psline[linewidth=0.4pt,linecolor=brown]{->}(0,-2.2)(0,2.75)
      \rput(-0.25,2.75){$y$}
        \rput(-0.16,2.17){$c$}
        \rput(-0.19,1.36){\scalebox{0.9}{$c_1$}}
        \rput(-0.2,0.8){\scalebox{0.9}{$c_2$}}
      \pspolygon[linecolor=blue,fillstyle=solid,fillcolor=lightgray](1.2,1.2)(2,2)(1.2,2)
       \pspolygon[linecolor=blue,fillstyle=solid,fillcolor=lightgray](1.2,1.2)(2,1.2)(2,0)(1.2,0)
\psline[linecolor=blue](1.2,0.8)(2,0.8)
\psline[linecolor=blue,linestyle=dashed,linewidth=0.4pt](0,0.8)(1.2,0.8)
\psline[linecolor=blue](1.2,0.6)(2,0.6)
\psline[linecolor=blue](1.2,0.44)(2,0.44)
\psline[linecolor=blue](1.2,0.32)(2,0.32)
      \rput(2.17,-0.17){$c$}
      \rput(-2.34,-0.17){$-c$}
      \rput(1,-0.2){$c_1$}
      \rput(-0.9,-0.15){\scalebox{0.8}{$-c_1$}}
  }
   \end{picture}
\caption{The frame of the construction in the $xy$-projection.}
\label{fig:frame}
\end{figure}

Due to symmetry of the construction, it suffices to consider the cases where the projection belongs to (i)  the triangle $c_1 < x < y < c$ and (ii) the rectangle $c_1 < x < c$,\, $0 < y < c_1$ (they are shown grey in Fig. \ref{fig:frame}).

Consider the case (i). Take an auxiliary trajectory
corresponding to the particle with the same initial data making
reflections from $\tilde\AAA_{yz}$. At the first point of
reflection one has $z = p_0(y)$, at the second point, $z =
q_1(y)$, and the third and fourth reflection points are
symmetric to the first two points with respect to the
$xy$-plane. After the fourth reflection the particle moves
freely. We know that $\tilde\AAA_{yz}$ is invisible in the
$z$-direction; that is, the initial and final parts of the
auxiliary trajectory belong to a straight line. It remains to
show that the points of impact actually belong to $B_{yz}$ and
the auxiliary trajectory does not intersect the rest of the
body $B \setminus B_{yz}$; this will imply that it is a {\it
true} billiard trajectory in the complement of $B$.

The parts of the auxiliary trajectory in $Q$ before the first reflection, between the first and the second reflection, and between the second and the third reflection will be referred to as $01$-segment, $12$-segment, and $23$-segment, respectively. At the first reflection point $(x^{(1)}, y^{(1)}, z^{(1)})$ one has
$$
z^{(1)} = p_0(y^{(1)}) > y^{(1)} > x^{(1)},
$$
therefore this point belongs to the interior of $\Pi_z$, and so, $(x^{(1)}, y^{(1)}, z^{(1)}) \in B_{yz}$.  At the second reflection point $(x^{(2)}, y^{(2)}, z^{(2)})$ one has $x^{(2)} = x^{(1)}$,\, $0 < y^{(2)} < y^{(1)}$, and $z^{(2)} > z^{(1)}$, therefore this point also belongs to the interior of $\Pi_z$, and so, $(x^{(2)}, y^{(2)}, z^{(2)}) \in B_{yz}$. Since $\Pi_z$ is convex, one concludes that the segment with endpoints $(x^{(1)}, y^{(1)}, z^{(1)})$ and $(x^{(2)}, y^{(2)}, z^{(2)})$ also belongs to the interior of $\Pi_z$.

Using (\ref{bxz}) and the inequality $c_1 < x^{(1)} < c$, one concludes that the intersection of $B_{xz}$ with the plane $x = x^{(1)}$  belongs to the set $\{ c_1 \le |y| \le c,\ z \le p_0(x^{(1)}) \}$. On the other hand, $p_0(x^{(1)}) < p_0(y^{(1)})$ and the $01$- and $12$-segments belong to the set $\{ z \ge p_0(y^{(1)}) \}$, and the $23$-segment belongs to the set $\{ |y| < c_1 \}$. This implies that the first three segments of the trajectory do not intersect $B_{xz}$, and due to the symmetry of both $B_{xz}$ and the trajectory with respect to the $xy$-plane, this is true for the whole trajectory.

Further, the $01$- and $12$-segments belong to the interior of $\Pi_z$, and therefore, do not intersect $B_{x}$ and $B_y$. Due to the symmetry with respect to the $xy$-plane, this is also true for the last two segments of the auxiliary trajectory. It remains therefore to prove that the $23$-segment does not intersect the bodies $B_{x} = B_{yx} \cup B_{zx}$ and $B_y = B_{zy} \cup B_{xy}$.

The sets $B_{zx}$,\, $B_{zy}$ and $B_{xy}$ belong to the galleries $G_z$ and $G_x$, and therefore cannot have points in common with the    $23$-segment. It remains to check the set $B_{yx}$.

At the second reflection point $(x^{(2)}, y^{(2)}, z^{(2)})$ one has $z^{(2)} = q_1(y^{(2)})$ and $z^{(2)} > x^{(2)}$. At each point of the $23$-segment one has $x = x^{(2)}$,\, $y = y^{(2)}$, and therefore, $x < q_1(y)$ and $c_2 < y < c_1$. By (\ref{byxincl}), no such point belongs to $B_{yx}$.

Consider now the case (ii). Take $i \ge 1$ such that $c_{i+1} < y < c_i$. (The limiting case $y = c_i$ has zero measure in the space of billiard trajectories; in this case the particle hits a singular point of $B$ and the motion is not defined since then.)

If $x > p_i(y)$, the particle does not hit $B$. Indeed, by (\ref{byzincl}), the vertical straight line $(x,y,\ast)$ does not intersect $B_{yz}$, and by (\ref{byxincl}), it does not intersect $B_{yx}$. The other sets $B_{xy},\, B_{xz},\, B_{zx},\, B_{zy}$ comprising $B$ belong to the galleries $G_x$ and $G_z$, and therefore do not intersect this straight line.

Assume that $x < p_i(y)$ and consider an auxiliary trajectory
with the same initial data making reflections from
$\tilde\AAA_{yz}$. As in the case (i), there are 4 reflections;
the first three segments of the trajectory (between the point
of entering $Q$ and the 1$^{\text{st}}$ reflection point;
between the 1$^{\text{st}}$ and the 2$^{\text{nd}}$ reflection
points; between the 2$^{\text{nd}}$ and the 3$^{\text{rd}}$
reflection points) will be referred to as $01$-segment,
$12$-segment, and $23$-segment, respectively. The trajectory is
symmetric with respect to the $xy$-plane, and the final
velocity, $(0,0,-1)$, coincides with the initial one.

Repeating the argument of (i), one concludes that the
1$^{\text{st}}$ and 2$^{\text{nd}}$ reflection points and the
$12$-segment joining them belong to $\Pi_z$, and by symmetry,
the 3$^{\text{rd}}$ and 4$^{\text{th}}$ reflection points, and
the segment joining them, also belong to $\Pi_z$. Therefore all
the reflection points belong to $B_{yz}$. It remains to check
that the auxiliary trajectory does not intersect $B \setminus
B_{yz}$, and therefore, is a {\it true} trajectory in the
complement of $B$.

Recall that the sets $B_{xy},\, B_{xz},\, B_{zx},\, B_{zy}$ belong to $G_x \cup G_z$, and therefore do not intersect the trajectory. It remains to check the set $B_{yx}$. Since the $01$- and the $12$-segments belong to $\Pi_z$, they do not have points in common with $B_{yx}$. The same is true for the last two segments symmetric to them. It remains to check the $23$-segment.

Let $(x^{(1)}, y^{(1)}, z^{(1)})$ and $(x^{(2)}, y^{(2)},
z^{(2)})$ be the points of first and second reflection. Notice
that the $23$-segment belongs to the straight line $(x^{(2)},
y^{(2)}, \ast)$. One has \beq\label{f1} c_{i+2} < y^{(2)} <
c_{i+1}, \eeq $x^{(2)} = x^{(1)} = x$,\, $0 < y^{(2)} < y^{(1)}
= y$,\, $p_i(y^{(1)}) < q_{i+1}(y^{(2)})$, and therefore,
\beq\label{f2} x^{(2)} < q_{i+1}(y^{(2)}). \eeq By
(\ref{byxincl}),(\ref{f1}),  and (\ref{f2}), the straight line
$(x^{(2)}, y^{(2)}, \ast)$ does not intersect $B_{yx}$. Hence
the theorem is proved.
\end{proof}

\begin{figure}[h]
\centering
\begin{picture}(440,220)
\put(0,30){\includegraphics[height=190pt,
keepaspectratio]{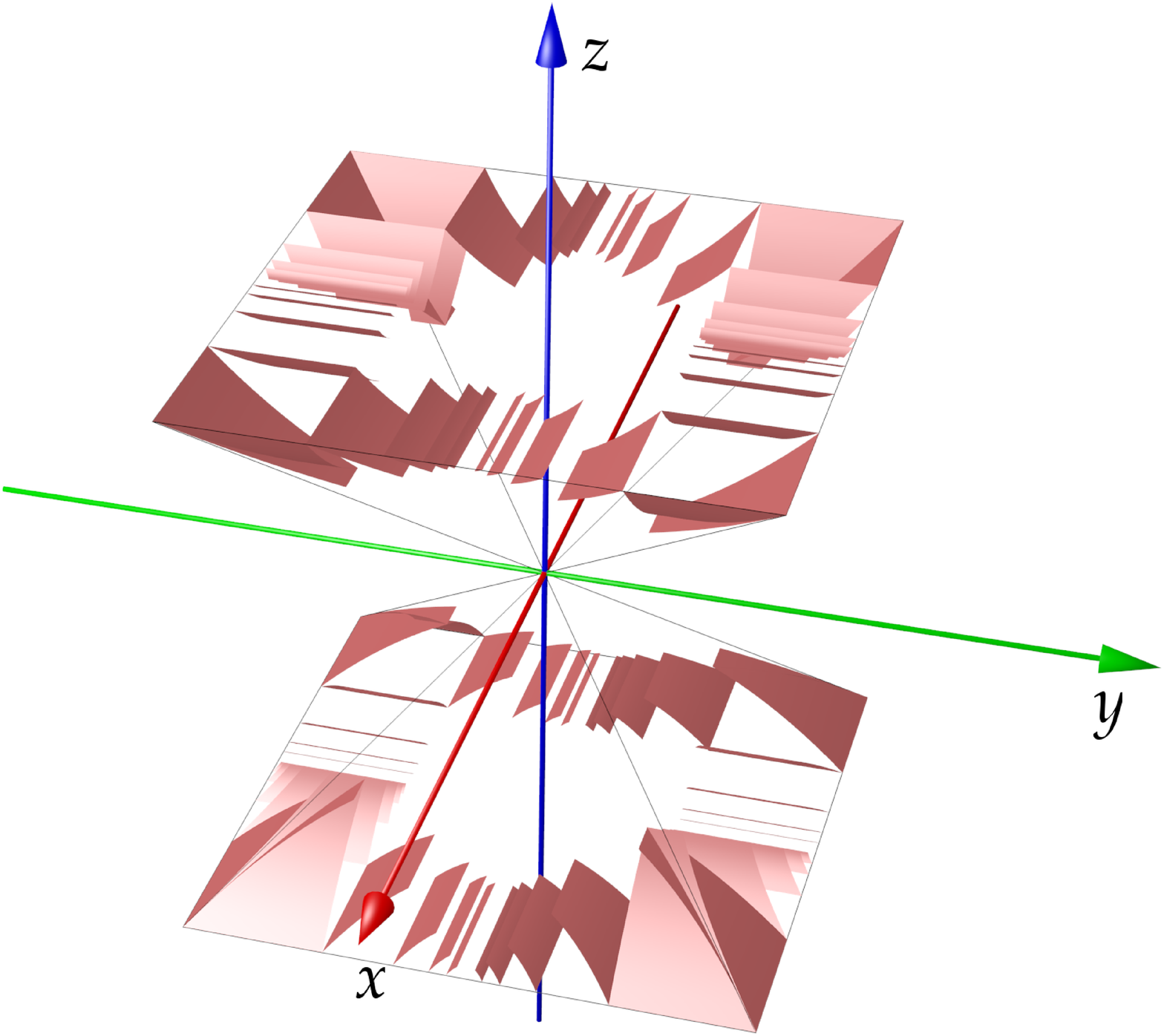}}
\put(220,30){\includegraphics[height=190pt,
keepaspectratio]{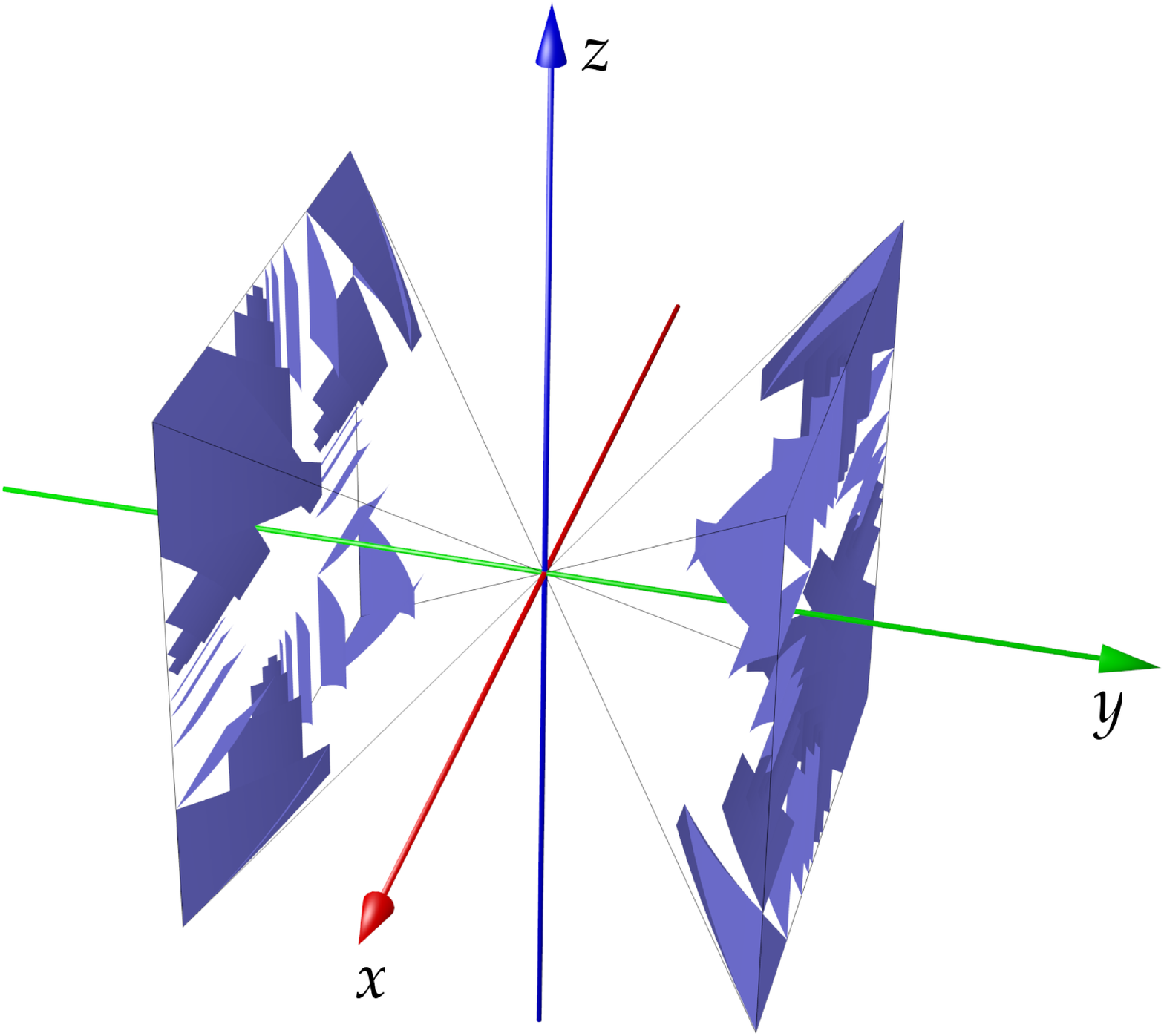}}
\put(60,10){(a)}
\put(260,10){(b)}
\end{picture}
\caption{Non-overlapping bodies invisible in different directions: (a) along the $z$-axis; (b) along the $y$-axis}\label{fig:3d04}
\end{figure}

\begin{figure}[h]
\centering
\begin{picture}(440,220)
\put(0,30){\includegraphics[height=190pt,
keepaspectratio]{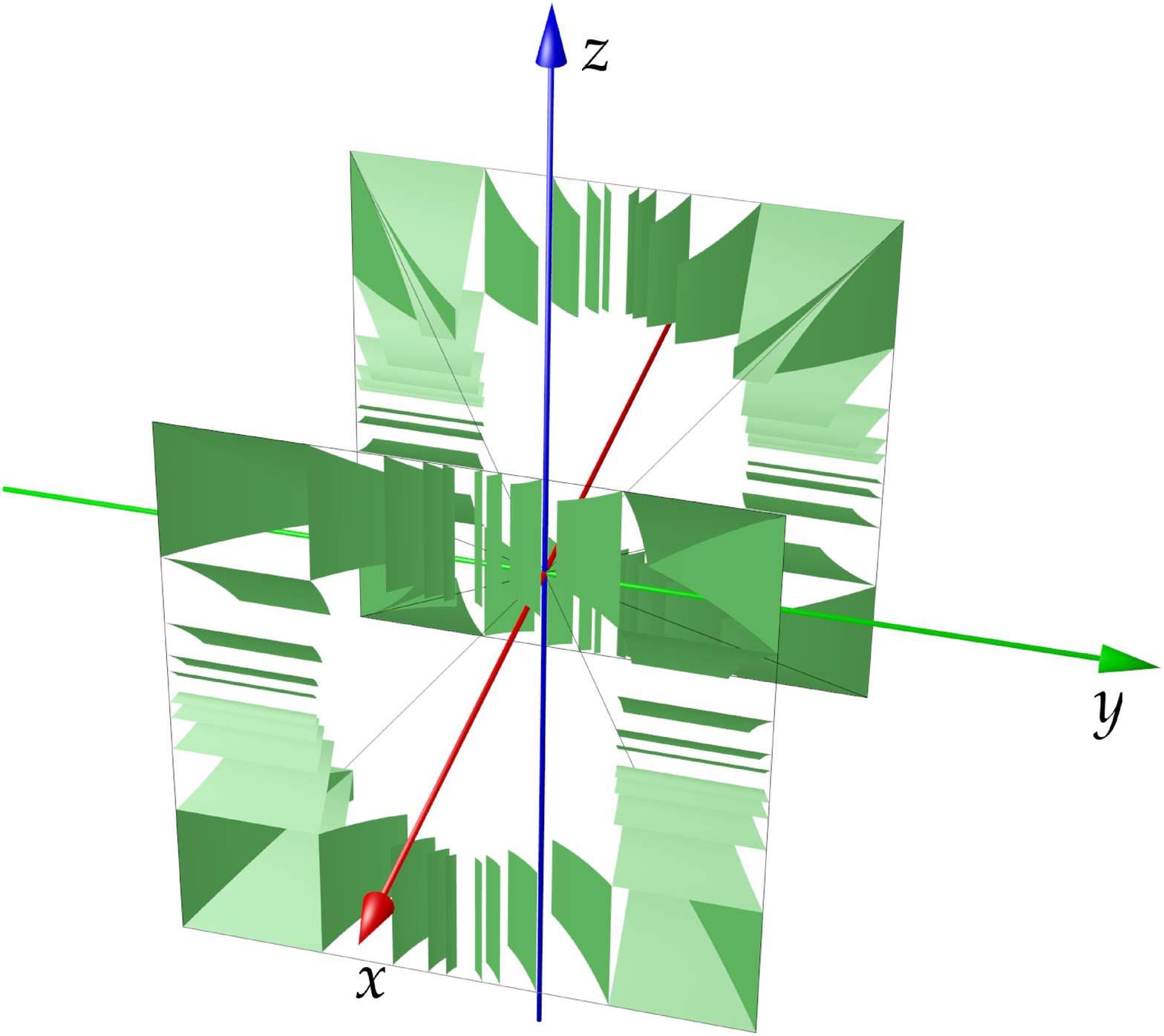}}
\put(220,30){\includegraphics[height=190pt,
keepaspectratio]{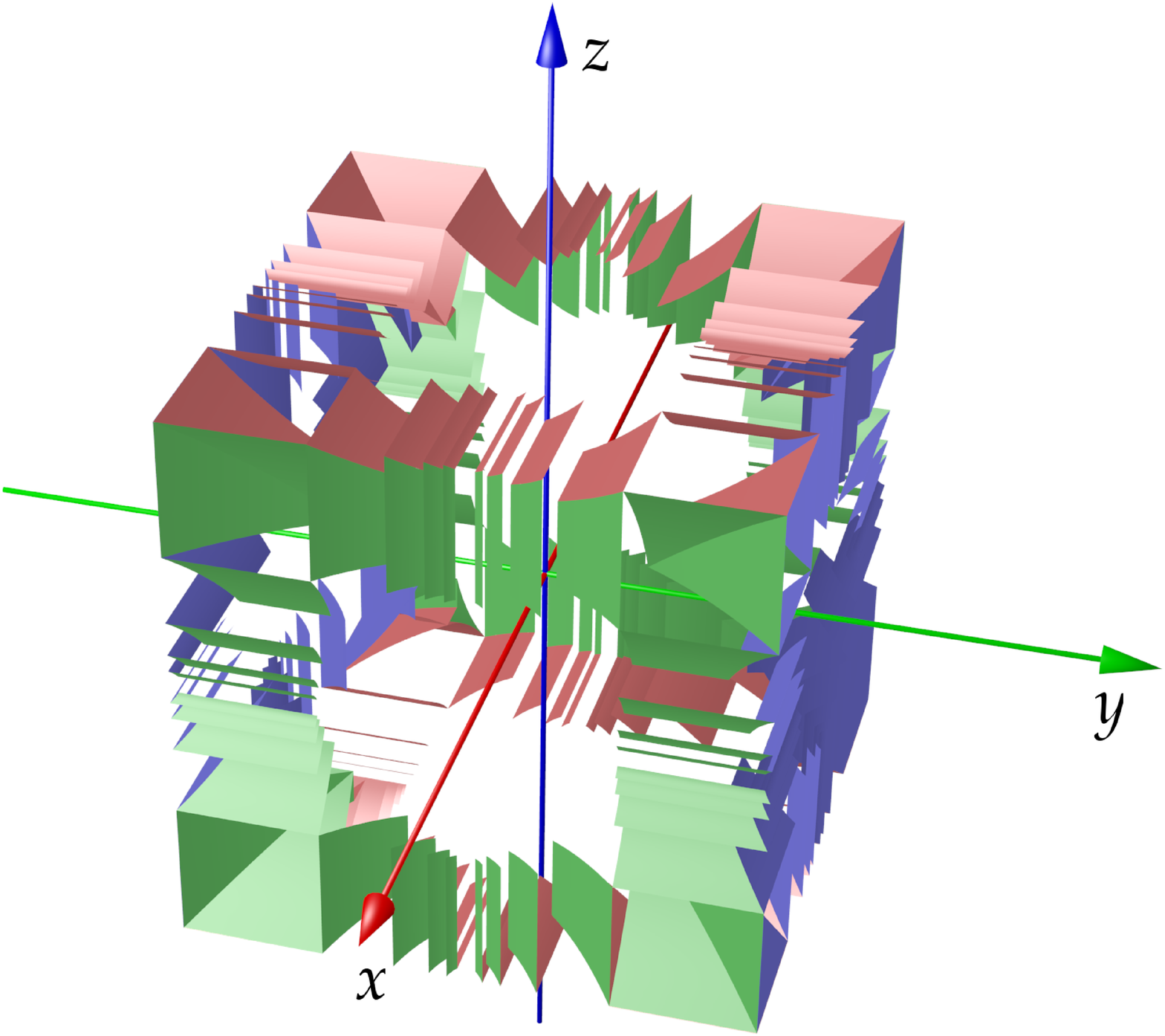}}
\put(60,10){(a)}
\put(260,10){(b)}
\end{picture}
\caption{A body invisible in the direction along the $x$-axis (a) and a body invisible in 3 directions (b)}\label{fig:3d05}
\end{figure}

\begin{remark}\label{rem:4}
We do not know if it is possible to generalize our construction to 3 {\rm non-orthogonal} directions.  The direct generalization does not work even in the case where two non-orthogonal directions lie in the $xy$-plane and the third one coincides with the $z$-direction. Indeed, a particle falling vertically down and hitting a mirror in the gallery $G_y$, will then move in a direction orthogonal to the $y$-axis (and not parallel to the $x$-axis, which would be desirable), and therefore may fail to hit the opposite mirror in that gallery.
\end{remark}

\section{Summary}\label{s:summary}

We have shown that there exist bodies invisible in 2 directions
in the two-dimensional case and in 3 directions in the
three-dimensional case. It was not known earlier whether such
bodies exist. We believe that our construction can be more or
less directly generalized to $n$ directions of invisibility for
$n$-dimensional bodies, $n > 3$.

There are, however, many open questions. Can we construct a
body invisible in $n$ directions in $n$-dimensional space
without using any fractal constructions? Are there bodies
invisible in more than $n$ directions in $n$-dimensional space?
What is the maximal number of directions of invisibility? How
to introduce an adequate ``measure of invisibility'' for a body
observed in all directions and find the ``most invisible''
body?

There is an intriguing observation related to the existing
constructions. There exist connected (and even homeomorphic to
the ball) bodies invisible in 1 direction \cite{0-resist}. The
body invisible in 2 directions found in
\cite{PlakhovRoshchina2011} is disconnected. The body invisible
in 3 directions has an infinite number of connected components.
We wonder if the increased complexity of the shape is the cost
one should pay for the increased number of directions, or
whether it is just an artefact of the particular constructions.

These problems are easy to understand, and the existing results
can be explained by using only basic school math. However,
there are no {\em tools} or {\em techniques} for constructing
invisible bodies, and this makes the subject even more
exciting.

\section*{Acknowledgements}

This work was partly supported by the {Center for Research and
Development in Mathematics and Applications (CIDMA)} from the
''{\it Funda\c{c}\~{a}o para a Ci\^{e}ncia e a Tecnologia}''
(FCT), cofinanced by the European Community Fund FEDER/POCTI,
and by the FCT research project PTDC/MAT/113470/2009.

\end{document}